%% file: AziMV20_revised2.tex
\documentclass[final]{siamart171218}
\usepackage{hyperref}
\usepackage{epsfig}
\usepackage{amsfonts}
\usepackage{algorithm}
\usepackage{algorithmic}
\usepackage{graphicx}
\usepackage{subfigure}
\usepackage{array}
\usepackage{enumerate}
\usepackage{enumitem}
\usepackage[T1]{fontenc}

\usepackage{pgfplots,tikz}

\newcommand {\cR}  {\mathcal{R}}

\DeclareMathOperator{\rank}{rank}
\DeclareMathOperator{\Real}{Re}
\DeclareMathOperator{\Imag}{Im}
\DeclareMathOperator{\Rs}{Rs}

\newcommand\Item[1][]{%
  \ifx\relax#1\relax  \item \else \item[#1] \fi
  \abovedisplayskip=0pt\abovedisplayshortskip=0pt~\vspace*{-1.25\baselineskip}}

\newcommand {\C}        {{\mathbb{C}}}

\newcommand{\sn}{\mathsf{n}}
\newcommand{\sk}{\mathsf{k}}
\newcommand{\sm}{\mathsf{m}}
\renewcommand{\sp}{\mathsf{p}}
\newcommand{\sd}{\mathsf{d}}
\newcommand{\sr}{\mathsf{r}}
\newcommand{\sq}{\mathsf{q}}
\DeclareMathOperator{\Ran}{Ran}

\newtheorem{assumption}[theorem]{Assumption}

\xdefinecolor{lavendar}{rgb}{0.4,0.2,0.6}
\xdefinecolor{lavender}{rgb}{0.92,0.9,0.97}
\xdefinecolor{paleblue}{rgb}{0.3,0.3,0.7}
\xdefinecolor{darkblue}{rgb}{0.15,0.15,0.9}
\xdefinecolor{mygreen}{rgb}{0.1,0.6,0.1}
\xdefinecolor{mydarkgreen}{rgb}{0.1,0.4,0.1}

\usepackage{todonotes}

\headers{R. Aziz, E. Mengi, and M. Voigt}{Interpolatory Frameworks for Eigenvalue Problems}

\title{Derivative Interpolating Subspace Frameworks for Nonlinear Eigenvalue Problems}
\author{
Rifqi Aziz\thanks{Ko\c{c} University, Department of Mathematics, Rumeli Feneri Yolu 34450, Sar{\i}yer, Istanbul, Turkey, E-Mail: \texttt{raziz14@ku.edu.tr}.}  \and
Emre Mengi\thanks{Ko\c{c} University, Department of Mathematics, Rumeli Feneri Yolu 34450, Sar{\i}yer, Istanbul, Turkey, E-Mail: \texttt{emengi@ku.edu.tr}.} \and
Matthias Voigt\thanks{Techni\-sche Universit\"at Berlin, Institut f\"ur Mathematik, Stra{\ss}e des 17. Juni 136, 10623 Berlin, Germany, E-Mail: \texttt{mvoigt@math.tu-berlin.de}.}
}

\begin{document}
\maketitle

\begin{abstract}
We first consider the problem of approximating a few eigenvalues of a rational matrix-valued 
function closest to a prescribed target. It is assumed that the proper rational part of the rational 
matrix-valued function is expressed in the transfer function form $H(s) = C (sI - A)^{-1} B$,
where the middle factor is large, whereas the number of rows of $C$ and the number of 
columns of $B$ are equal and small. We propose a subspace framework that performs 
two-sided or one-sided projections on the state-space representation of $H(\cdot)$, commonly employed in 
model reduction and giving rise to a reduced transfer function. At every iteration, the projection 
subspaces are expanded to attain Hermite interpolation conditions at the eigenvalues of the 
reduced transfer function closest to the target, which in turn leads to a new reduced transfer 
function. We prove in theory that, when a sequence of eigenvalues of the reduced transfer functions
converges to an eigenvalue of the full problem, it converges at least at a quadratic rate. 
In the second part, we extend the proposed framework to locate the eigenvalues of a general 
square large-scale nonlinear meromorphic matrix-valued function $T(\cdot)$, where we 
exploit a representation $\cR(s) = C(s) A(s)^{-1} B(s) - D(s)$ defined in terms of the block 
components of $T(\cdot)$. The numerical experiments illustrate that the proposed framework 
is reliable in locating a few eigenvalues closest to the target point, and that, with respect to runtime, 
it is competitive to established methods for nonlinear eigenvalue problems.
\end{abstract}



\section{Introduction}
The numerical solutions of nonlinear eigenvalue problems have been a major field of research in the last twenty years \cite{Mehrmann2004, Guttel2017}. Numerical algorithms are proposed to estimate the eigenvalues 
of a nonlinear matrix-valued function either within a prescribed region, or closest to a prescribed target point in 
the complex plane. 

Earlier works are mostly focused on polynomial and rational eigenvalue problems \cite{Tisseur2001, Mackey2006, Su2011}. 
More recently, some of the attention has shifted to nonlinear eigenvalue problems that are neither polynomial nor rational. 
Various applications give rise to such non-polynomial, non-rational eigenvalue problems,
including the stability analysis of delay systems \cite{Guttel2017}, numerical solutions of elliptic 
PDE eigenvalue problems by the boundary element method \cite{Effenberger2012}, or finite element 
discretizations of differential equations with nonlinear boundary conditions depending on
an eigenvalue parameter \cite{Betcke2012}.

The nonlinear eigenvalue problem setting that we consider in this work is as follows. Let 
\begin{equation}\label{eq:NEP_int0}
	T(s) := f_1(s) T_1 + \dots + f_{\kappa}(s) T_\kappa,
\end{equation}
where the functions $f_1,\,\dots,\,f_\kappa : {\mathbb C} \rightarrow {\mathbb C}$ are meromorphic,
and $T_1,\, \dots,\, T_\kappa \in {\mathbb C}^{\sn \times \sn}$ are given matrices. 
Assume that the set $ \{ \lambda \in \C \; | \; \rank_{\lambda \in \C} T(\lambda) < \sn \}$ consists 
only of isolated points. Then we want to find $\lambda \in \C$ and $v \in \C^\sn \setminus \{0\}$ 
such that
\begin{equation}\label{eq:NEP_int}
	T(\lambda) v = 0.
\end{equation}
The scalar $\lambda \in {\mathbb C}$
satisfying \eqref{eq:NEP_int} is called an eigenvalue, and the vector $v \in {\mathbb C}^\sn \setminus \{0\}$ is
called a corresponding eigenvector. This setting is quite general. For instance, polynomial and
rational eigenvalue problems are special cases when $f_j(\cdot)$ are scalar-valued polynomials and
rational functions, respectively. Delay eigenvalue problems can also be expressed in this form
such that some of $f_j(\cdot)$ are exponential functions.

We propose an interpolation-based subspace framework to find a prescribed number of eigenvalues
of $T(\cdot)$ closest to a given target point $\tau \in {\mathbb C}$. At every iteration, a projected 
small-scale nonlinear eigenvalue problem is solved. Then the projection subspaces are expanded 
so as to satisfy Hermite interpolation properties at the eigenvalues of the projected problem. 
The projections we rely on are devised from two-sided or one-sided projections commonly employed in 
model-order reduction \cite{Gugercin2009}.
 
Our approach could be compared with linearization-based techniques for nonlinear eigenvalue
problems. 
However, such techniques first use either polynomial interpolation \cite{Effenberger2012, Beeumen2013}
or rational interpolation \cite{Guttel2014} to approximate the nonlinear matrix-valued function
with a polynomial or a rational matrix-valued function. Then the polynomial and rational eigenvalue
problems are linearized into generalized eigenvalue problems. In order to deal with large-scale
problems, typically Krylov subspace methods are applied to the generalized eigenvalue problem
in an efficient manner, in particular taking the structure of the linearization into account;
see for instance \cite{VanBeeumen2015} and \cite{Lietaert2018} for a one-sided and a two-sided
rational Arnoldi method, respectively.
Further advances in approximating a nonlinear eigenvalue problem by a rational one have been proposed very recently. In the recent works \cite{GutNT20,LieMPV21}, certain variants of the Antoulas-Anderson algorithm (AAA) are investigated for this purpose. Also \cite{GutNT20} gives a detailed error analysis, i.\,e., the authors investigate how the quality of the rational approximation affects the quality of the computed eigenvalues. Based on this analysis, they propose a new termination condition. Another recent work \cite{BreEM20} combines data-driven interpolation approaches by the Loewner framework with contour-integral methods to compute all eigenvalues inside a specified contour. 

The approach proposed here is somewhat related to \cite{BreEM20} since we also use interpolation, but otherwise it differs from the above-mentioned works.
We apply subspace projections directly to the nonlinear eigenvalue problem (and hence, preserve its nonlinear structure), and the
projection subspaces are not necessarily Krylov subspaces. Consequently, our approach
assumes the availability of numerical techniques for the solutions of the projected small-scale
nonlinear eigenvalue problems. In the case of polynomial or rational eigenvalue problems,
linearization based techniques are available to our use to obtain all of the eigenvalues
of the small-scale problem. Our numerical experience is that the proposed frameworks
here are comparable to the state-of-the-art methods in terms of computational efficiency,
and even a few times faster in some cases.


\textbf{Outline.} In the next section, we first describe an interpolatory subspace framework specifically for
rational eigenvalue problems. For instance, for a proper rational matrix-valued function,
which can always be expressed in the form $R(s) = C(s I_{\sk} - A)^{-1} B$ for some $A \in {\mathbb C}^{\sk \times \sk}$, $B \in {\mathbb C}^{\sk \times \sn}$, and 
$C \in {\mathbb C}^{\sn \times \sk}$, the framework addresses the case when $\sk \gg \sn$
and reduces the dimension of the middle factor. We give formal arguments
establishing the at least quadratic convergence of the proposed subspace framework.
In Section \ref{sec:NEP}, we extend the subspace framework idea for rational
eigenvalue problems to the general nonlinear eigenvalue problem setting of \eqref{eq:NEP_int}.
Sections \ref{sec:REP} and \ref{sec:NEP} present the frameworks to locate
only one eigenvalue closest to the prescribed target, and employ two-sided projections. Section \ref{sec:multiple_eig}
discusses how the frameworks can be generalized to locate a prescribed number of closest
eigenvalues to the target, while Section \ref{sec:one_sided} describes how one-sided
projections can be adopted in place of two-sided projections.
Finally, Section \ref{sec:num_results} illustrates the frameworks on synthetic examples, as well as
classical examples from the NLEVP data collection \cite{Betcke2012}, and confirms the validity of the theoretical findings in practice.

\section{Rational Eigenvalue Problems}\label{sec:REP}
A special important class of nonlinear eigenvalue problems are rational eigenvalue problems. 
There we want to find $\lambda \in \C$ and $v \in \C^\sn \setminus \{0\}$ such that
\begin{equation}\label{eq:REP}
	R(\lambda) v = 0,
\end{equation}
where $R(\cdot)$ is a rational matrix-valued function. We again assume that the set 
$ \{ \lambda \in \C \; | \; \rank_{\lambda \in \C} R(\lambda) < \sn \}$ consists only of isolated points.
The significance of the rational eigenvalue problem 
is due to several reasons. First, it  is a cornerstone for the solutions of nonlinear eigenvalue 
problems that are not rational; such nonlinear eigenvalue problems are often approximated by rational 
eigenvalue problems. Secondly, there are several applications that give rise to rational eigenvalue problems 
such as models for vibrations of fluid-solid structure, as well as vibrating mechanical structures \cite{Mehrmann2004}.

There are several ways to represent the function $R(\cdot)$. One possibility is to write it as
\begin{equation}\label{eq:REP_lform}
	R(s) = P(s) +  \sum_{j=1}^{\rho} \frac{p_j(s)}{d_j(s)}  E_j,
\end{equation}
where $d_j,\, p_j : {\mathbb C} \rightarrow {\mathbb C}$ are polynomials of degree ${\sk}_j$ and strictly less 
than ${\sk}_j$, respectively. Moreover, $E_j \in {\mathbb C}^{\sn \times \sn}$ for $j = 1,\,\dots,\, \rho$ are given 
matrices, and $P(\cdot)$ is a matrix polynomial of degree $\sd$ of the form
$
	P(s)
		:=
	s^\sd P_\sd  +  \dots  +  s P_1  +  P_0
$
for given $P_0,\, \dots,\, P_\sd \in {\mathbb C}^{\sn\times \sn}$.

It is usually the case that the matrices $E_j$ in \eqref{eq:REP_lform} are of low rank, hence they can be decomposed into
\begin{equation}\label{eq:lowrank_dc}
		E_j
			=
		L_j U_j^\ast
\end{equation}
for some $L_j,\, U_j \in {\mathbb C}^{\sn\times {\sr}_j}$ of full column rank such that ${\sr}_j \ll {\sn}$. In this case,
the proper rational part of $R(s)$ can always be expressed as a transfer function associated with
a linear time-invariant system. Formally, it can be shown that
\begin{equation}\label{eq:prop_rational}
		\sum_{j=1}^\rho
			 \frac{p_j(s)}{d_j(s)}  E_j
				  = 
			C ( sI_{\sk} - A )^{-1} B
\end{equation}
in (\ref{eq:REP_lform}) for some $A \in {\mathbb C}^{\sk \times \sk}$, $B \in {\mathbb C}^{\sk \times \sn}$, 
$C \in {\mathbb C}^{\sn \times \sk}$, where ${\sk} := {\sr}_1 {\sk}_1 + \dots + {\sr}_{\rho} {\sk}_{\rho}$. 
We refer to \cite{Su2011} and \cite[page 95]{Antoulas2005} for the details of the construction
of $A, B, C$ from the polynomials $p_j, d_j$ in (\ref{eq:REP_lform}) and matrices $L_j, U_j$ in (\ref{eq:lowrank_dc}).

The expression of the rational eigenvalue problem in the form
\begin{equation} \label{eq:TFform}
   R(s) = P(s) + C(sI_{\sk} - A)^{-1} B
\end{equation}
(or more generally, $R(s) = P(s) + C(s)D(s)^{-1} B(s)$ with matrix polynomials 
$P(\cdot)$, $A(\cdot)$, $B(\cdot)$, and $C(\cdot)$) is sometimes immediately available. Indeed,
this is another way to formulate a rational eigenvalue 
problem. This formulation plays a more prominent role in applications from linear systems and control theory. For example, the eigenvalues of the rational function $R(s) = C(sI_{\sk} - A)^{-1}B$ are 
the so-called \emph{transmission zeros} of the linear time-invariant system 
\begin{equation} \label{eq:full_sys}
    \frac{\mathrm{d}}{\mathrm{d}t} x(t) = A x(t) + B u(t), \quad y(t) = Cx(t), 
\end{equation}
see for example \cite{FraW75}. The transmission zeros play prominent roles in electronics applications such as in oscillation damping control \cite{MarPR07} or the design of filters \cite{QueRP03}. In the latter applications, the transmission zeros are used to specify frequency bands for which input signals are rejected.

One way of dealing with \eqref{eq:REP} is to convert it into a generalized eigenvalue problem. 
For instance, for any $\sk \times \sk$ matrix $F$, we have
\begin{equation}\label{eq:proper_REP_lin}
	C (\lambda F -  A)^{-1} B v = 0		\quad \Rightarrow \quad
	\left(	
		\begin{bmatrix}
			A	&	B	\\
			C	&	0	\\
		\end{bmatrix}		
			-
	\lambda
		\begin{bmatrix}
			F	&	0	\\
			0	&	0	\\
		\end{bmatrix}
	\right)
		\begin{bmatrix}
			(\lambda F - A)^{-1} B v	\\
			v
		\end{bmatrix}
=	0.
\end{equation}
More generally, $\left\{ P(\lambda) +  C (\lambda F -  A)^{-1} B \right\} v = 0$
for $P(s) = \sum_{j = 0}^{\sd} s^j P_j$ and for any $\sk \times \sk$ matrix $F$ can be linearized into
\begin{equation}\label{eq:linear_EP}
\begin{split}
	({\mathcal A} - \lambda {\mathcal B}) z = 0, \qquad	\text{where}	\quad
	z
		:=
	\left[
		\begin{array}{c}
			(\lambda F - A)^{-1} B v	\\
			\hline
				\lambda^{\sd-1} v		\\
				\lambda^{\sd-2} v		\\
				\vdots		\\
				v
		\end{array}
	\right],
		\\
	{\mathcal A}
		:=
	\left[
		\begin{array}{c|cccc}
			A	&		&		&		&	B	\\
			\hline
			C	& P_{\sd-1}	& P_{\sd-2}	&	\dots	&	P_0	\\
				& -I_\sn		&	0	&	\dots	&	0	\\
				&		& \ddots	& \ddots	&	\vdots	\\
				&		&		&	-I_\sn	&	0			\\
		\end{array}
	\right],	\;\;
	{\mathcal B}
		:=
	\left[
		\begin{array}{c|cccc}
			F	&		&		&		&		\\
			\hline
				& -P_{\sd}	& 		&		&		\\
				& 		&	-I_\sn	&		&		\\
				&		& 		& \ddots	&		\\
				&		&		&		&	-I_\sn			\\
		\end{array}
	\right],
\end{split}
\end{equation}
where we set $\sd = 1$ and $P_1 = 0$ in the case $P(s) \equiv P_0$.
For $F = I_{\sn}$, there is a one-to-one correspondence between the 
eigenvalues of $L(s) := {\mathcal A} - s {\mathcal B}$ defined as in (\ref{eq:linear_EP})
and $R(\cdot)$ as in (\ref{eq:REP}). In particular, for this choice of $F$ the following holds:
\begin{itemize}
	\item If $\lambda$ is an eigenvalue of the pencil $L(s) := {\mathcal A} - s {\mathcal B}$,
	but not an eigenvalue of $A$, then $\lambda$ is also an eigenvalue of $R(\cdot)$ \cite{Su2011}. 
	\item Conversely, if $\lambda$ is an eigenvalue of $R(\cdot)$, it is also an eigenvalue of 
	$L(\cdot)$; see also \cite{Rommes2006}.
\end{itemize}

\subsection{The Subspace Method for Rational Eigenvalue Problems}
The setting we aim to address in this section is when the rational eigenvalue problem \eqref{eq:REP} is given in the transfer
function form \eqref{eq:TFform}, where the size of the matrix $A$ is very large compared to ${\sn} \cdot \sd$ 
(i.e., the size of $R(\cdot)$ $\cdot$ the degree of $P(\cdot)$), that is ${\sk} \gg {\sn} \cdot \sd$. In the
special case when  $P(s) \equiv P_0$, we aim to address the setting when $\sk \gg \sn$. 

Here, we propose a subspace framework that replaces the proper rational part $R_{\rm p}(s) := C (sI - A)^{-1} B$
of $R(\cdot)$ with a reduced one of the form
\[
	R^{{\mathcal W}, {\mathcal V}}_{\rm p} (s) := CV (s W^\ast V - W^\ast A V)^{-1} W^\ast B
\]
for two subspaces ${\mathcal W},\,{\mathcal V} \subseteq \C^\sk$ of equal dimension, say ${\sr}$ such that ${\sr} \ll {\sk}$, and matrices $W,\,V \in \C^{\sk \times {\sr}}$ whose columns form orthonormal bases for
the subspaces ${\mathcal W}$, ${\mathcal V}$, respectively. We remark that the ${\sr} \times {\sr}$ middle factor of the reduced 
proper rational function is much smaller than the ${\sk} \times {\sk}$ middle factor of the full problem.
The full rational function $R_{\rm p}(\cdot)$ and the reduced one $R^{{\mathcal W}, {\mathcal V}}_{\rm p}(\cdot)$ are
the transfer functions of the linear time-invariant systems (\ref{eq:full_sys}) and
\[
	\frac{\mathrm{d}}{\mathrm{d}t}W^\ast V x(t)  =  W^\ast A V x(t)  + W^\ast B u(t),	\quad	y(t) = C V x(t),
\]
respectively.
Hence, in the system setting, replacing $R_{\rm p}(\cdot)$ with $R^{{\mathcal W}, {\mathcal V}}_{\rm p}(\cdot)$
corresponds to restricting the state-space of \eqref{eq:full_sys} to ${\mathcal V}$, and then imposing a Petrov-Galerkin condition on the residual of the restricted state-space system to ${\mathcal W}$.

Our approach is interpolatory and inspired by model order reduction 
techniques \cite{Yousuff1985, DeVillemagne1987, Gugercin2013, Baur2014}, 
as well as by a recent subspace framework proposed for the estimation of the 
${\mathcal H}_{\infty}$ norm of a transfer function \cite{Aliyev2017}.  
The problem at hand \eqref{eq:REP} can be viewed as the minimization problem
\[
		\min_{\lambda \in {\mathbb C}} \sigma_{\min}( R(\lambda) ).
\]
Rather than this problem, at every iteration, we solve a reduced problem of the form
\begin{equation}\label{eq:red_problem}
		\min_{\lambda \in {\mathbb C}} 
			\sigma_{\min}\big( R^{{\mathcal W}, {\mathcal V}}(\lambda)\big),
\end{equation}
where $R^{{\mathcal W}, {\mathcal V}}(s):= P(s) + R^{{\mathcal W}, {\mathcal V}}_{\rm p}(s)$.
Then we expand the subspaces ${\mathcal W}, {\mathcal V}$ to $\widetilde{\mathcal W}, \widetilde{\mathcal V}$
so that
\begin{equation}\label{eq:Hermite_inter}
	R\big(\widetilde{\lambda}\big) = R^{\widetilde{\mathcal W}, \widetilde{\mathcal V}}\big(\widetilde{\lambda}\big)	\quad	\text{and}		\quad
			R'\big(\widetilde{\lambda}\big) = \left[ R^{\widetilde{\mathcal W}, \widetilde{\mathcal V}} \right]' \big(\widetilde{\lambda}\big)
\end{equation}
at a global minimizer $\widetilde{\lambda}$ of \eqref{eq:red_problem}. The procedure is repeated by solving
another reduced problem as in \eqref{eq:red_problem}, but with $\widetilde{\mathcal W}, \widetilde{\mathcal V}$ 
taking the role of ${\mathcal W}, {\mathcal V}$.

One neat issue here is that, recalling $L(s) = {\mathcal A} - s {\mathcal B}$ with ${\mathcal A}, {\mathcal B}$
as in (\ref{eq:linear_EP}) for $F = I_n$, finding the global minimizers of \eqref{eq:red_problem} amounts to computing
the eigenvalues of the pencil
\begin{equation}\label{eq:red_REP}
	\begin{split}
	L^{W, V}(s)
			& :=
		\begin{bmatrix}
			W^\ast	&	0	\\
				0	&	I_{\sn\sd}	\\
		\end{bmatrix}
	L(s)
		\begin{bmatrix}
			V	&	0	\\
			0	&	I_{\sn\sd}
		\end{bmatrix}
			=
	{\mathcal A}^{W,V}  - s {\mathcal B}^{W,V},	\quad \text{where}	\\
	{\mathcal A}^{W,V}
		& :=
	\left[
		\begin{array}{c|cccc}
			W^\ast A V	&		&		&		&	W^\ast B	\\
			\hline
			C V	& P_{\sd-1}	& P_{\sd-2}	&	\dots	&	P_0	\\
				& -I_\sn		&	0	&	\dots	&	0	\\
				&		& \ddots	& \ddots	&	\vdots	\\
				&		&		&	-I_\sn	&	0			\\
		\end{array}
	\right],	\\
	{\mathcal B}^{W,V}
		& :=
	\left[
		\begin{array}{c|cccc}
			W^\ast V	&		&		&		&		\\
			\hline
				& -P_{\sd}	& 		&		&		\\
				& 		&	-I_\sn	&		&		\\
				&		& 		& \ddots	&		\\
				&		&		&		&	-I_\sn			\\
		\end{array}
	\right],
	\end{split}
\end{equation}
which is immediate from (\ref{eq:linear_EP}) by replacing $A$, $B$, $C$, and $F$ with 
$W^\ast A V$, $W^\ast B$, $CV$, and $W^\ast V$, respectively. 
We remark that the pencil $L^{{W}, {V}}(\cdot)$ in \eqref{eq:red_REP} is of size $({\sr} + {\sn} \cdot \sd) \times ({\sr} + {\sn} \cdot \sd)$,
whereas the original pencil $L(s) = {\mathcal A} - s {\mathcal B}$ is of size $({\sk} + {\sn} \cdot \sd) \times ({\sk} + {\sn} \cdot \sd)$.  
As for the choice of at which eigenvalue $\widetilde{\lambda}$ of $L^{{W}, {V}}(\cdot)$
we would Hermite interpolate, we prescribe a target $\tau$ a priori, and choose $\widetilde{\lambda}$
as the eigenvalue of $L^{{W}, {V}}(\cdot)$ closest to $\tau$.

The only remaining issue that needs to be explained is how we expand the subspaces 
${\mathcal W}, {\mathcal V}$ into $\widetilde{\mathcal W}, \widetilde{\mathcal V}$ so as
to satisfy (\ref{eq:Hermite_inter}). Fortunately, the tools for this purpose have already been
established as elaborated in the following result. This result is an immediate corollary of 
\cite[Theorem 1]{Gugercin2009}.

\begin{lemma}\label{thm:main_interpolation}
Suppose that $\mu \in {\mathbb C}$ is not an eigenvalue of $A$.
Let $\widetilde{\mathcal W} = {\mathcal W} \oplus {\mathcal W}_\mu$ and 
$\widetilde{\mathcal V} = {\mathcal V} \oplus {\mathcal V}_\mu$, where ${\mathcal V}, {\mathcal W}$
are given subspaces of equal dimension, and ${\mathcal W}_\mu$, ${\mathcal V}_\mu$
are subspaces defined as 
		\[
			{\mathcal V}_\mu := \bigoplus_{j=1}^{\sq} \Ran \left((A - \mu I_{\sk})^{-j} B \right)
							\quad \text{and} \quad {\mathcal W}_\mu := \bigoplus_{j=1}^{\sq} \Ran \big( \big( C (A - \mu I_{\sk})^{-j} \big)^\ast \big)
		\]
for some positive integer ${\sq}$. Let $\widetilde{V}$ and $\widetilde{W}$ be basis matrices of $\widetilde{\mathcal{V}}$ and $\widetilde{\mathcal{W}}$, respectively and assume further that $\widetilde{W}^*A \widetilde{V}-\mu \widetilde{W}^*\widetilde{V}$ is invertible. Then we have
	\begin{enumerate}
		\item $R(\mu) = R^{\widetilde{\mathcal W}, \widetilde{\mathcal V}}(\mu)$, and		
		\item $R^{(j)}(\mu) = \left[ R^{\widetilde{\mathcal W}, \widetilde{\mathcal V}} \right]^{(j)}(\mu)\:$	
		for $j = 1,\,\dots,\, 2{\sq} - 1$,
	\end{enumerate}	
where $\: R^{(j)}(\cdot) \:$ and $\: \left[ R^{\widetilde{\mathcal W}, \widetilde{\mathcal V}} \right]^{(j)}(\cdot)\:$ denote the $j$th
derivatives of $R(\cdot)$ and $R^{\widetilde{\mathcal W}, \widetilde{\mathcal V}}(\cdot)$.
\end{lemma}

The resulting subspace method is described formally in Algorithm \ref{alg:SM}, where
we assume that the proper rational part of $R(\cdot)$ is provided as an input in the transfer function
form \eqref{eq:prop_rational} in terms of $A,\, B,\, C$.
At iteration $\ell$, the subspaces ${\mathcal V}_{\ell-1},\, {\mathcal W}_{\ell-1}$ are expanded
into ${\mathcal V}_{\ell},\, {\mathcal W}_{\ell}$
in order to achieve $R(\lambda_\ell) = R^{{\mathcal W}_\ell, {\mathcal V}_\ell}(\lambda_\ell)$ as well as
$R^{(j)}(\lambda_\ell) = \left[ R^{{\mathcal W}_\ell, {\mathcal V}_\ell} \right]^{(j)} (\lambda_\ell)$ for $j = 1,\,\dots,\, 2{\sq}-1$.
Lines \ref{exp_start}--\ref{exp_end} of the algorithm fulfill this expansion task by augmenting  
$V_{\ell-1}$, $W_{\ell-1}$, matrices whose columns form orthonormal bases for ${\mathcal V}_{\ell-1}$ and ${\mathcal W}_{\ell-1}$, with additional columns as suggested by Theorem \ref{thm:main_interpolation}. Orthonormalizing the augmented
matrices gives rise to the matrices $V_\ell$, $W_\ell$ whose columns span the expanded subspaces ${\mathcal V}_\ell$, ${\mathcal W}_\ell$, respectively.
The next interpolation point $\lambda_{\ell+1}$ is then set equal to the eigenvalue of 
$L^{{W}_{\ell},{V}_{\ell}}(\cdot)$ closest to the target point $\tau$.
Note that in line \ref{return_spec} of Algorithm \ref{alg:SM}, letting $\sr := {\rm dim} \: {\mathcal V}_\ell$,
the vector $v_{\ell+1}$ is of size $\sr + \sn\sd$, and $ v_{\ell+1}$\small$(\sr+\sn (\sd-1) + 1:\sr + \sn \sd)$, \normalsize
denoting the vector composed of the last $\sn$ entries of $v_{\ell+1}$, is an eigenvector estimate
according to (\ref{eq:linear_EP}).

\begin{algorithm}[tb]
 \begin{algorithmic}[1]
 
\REQUIRE{matrices $P_1,\,\dots,\,P_{\sd}  \in \mathbb{C}^{{\sn}\times {\sn}}$ as in \eqref{eq:REP}, 
	and $A \in {\mathbb C}^{\sk \times \sk}$, $B \in {\mathbb C}^{\sk \times \sn}$, $C \in {\mathbb C}^{\sn\times \sk}$  
	as in \eqref{eq:prop_rational}, interpolation parameter ${\sq} \in {\mathbb Z}$ with ${\sq} \geq 2$,
	target $\tau \in {\mathbb C}$.}
\ENSURE{estimate $\lambda \in \mathbb{C}$ for the eigenvalue closest to $\tau$ and corresponding eigenvector estimate $v \in \mathbb{C}^\sn \setminus \{0\}$.}

\vskip .7ex

\STATE $\lambda_0 \gets \tau$.

\vskip .7ex
\textcolor{mygreen}{\textbf{$\%$ main loop}}
\FOR{$\ell = 0,\,1,\,2,\,\dots$}
	
	  \STATE $\widehat{V}_\ell \gets (A - \lambda_\ell I)^{-1} B$,
		  		$\widetilde{V}_\ell \gets  \widehat{V}_\ell $, $\widehat{W}_\ell \gets (A - \lambda_\ell I)^{-\ast} C^\ast$,
				and $\widetilde{W}_\ell \gets \widehat{W}_\ell$.	\label{exp_start}	
			\FOR{$j = 2,\, \dots,\, {\sq}$}
		  \STATE $\widehat{V}_\ell \gets (A - \lambda_\ell I)^{-1} \widehat{V}_\ell$ and  
		  		$\widetilde{V}_\ell \gets \begin{bmatrix} \widetilde{V}_{\ell} & \widehat{V}_\ell \end{bmatrix}$.
		  \STATE $\widehat{W}_\ell \gets (A - \lambda_\ell I)^{-\ast} \widehat{W}_\ell$
						and
				$\widetilde{W}_\ell \gets \begin{bmatrix} \widetilde{W}_{\ell} & \widehat{W}_\ell \end{bmatrix}$.
		        \ENDFOR 	
	
	\vskip .7ex
		        
	\textcolor{mygreen}{\textbf{$\%$ expand the subspaces to interpolate at $\lambda_\ell$}}
	\IF{$\ell = 0$}
		\STATE $V_0 \gets \operatorname{orth}\left( \widetilde{V}_0 \right)
						\text{ and }  W_0 \gets \operatorname{orth}\left( \widetilde{W}_0 \right)$.
	\ELSE	        
		\STATE $V_\ell \gets \operatorname{orth}\left(\begin{bmatrix} V_{\ell-1} & \widetilde{V}_\ell \end{bmatrix}\right)
			 \text{ and } W_\ell \gets \operatorname{orth}\left(\begin{bmatrix} W_{\ell-1} & \widetilde{W}_\ell \end{bmatrix}\right)$. \label{R_orthogonalize}
	\ENDIF	\label{exp_end}
	\STATE Form $L^{W_{\ell},V_{\ell}}(s) := 
						{\mathcal A}^{W_{\ell}, V_{\ell}} - s {\mathcal B}^{W_{\ell}, V_{\ell}}$ 
			as in \eqref{eq:red_REP}.

	\vskip .7ex	
	\textcolor{mygreen}{\textbf{$\%$ update the eigenvalue estimate}}		
	\STATE  $\lambda_{\ell+1}, v_{\ell+1} \gets \text{ the eigenvalue, eigenvector of } L^{{W}_{\ell},{V}_{\ell}}(\cdot)$
			closest to $\tau$.	\label{inter_point}
	\vskip .2ex
	\STATE \textbf{Return} $\lambda \gets \lambda_{\ell+1}$,
	$\: v \gets \small v_{\ell+1}$\small$(\sr+\sn (\sd-1) + 1:\sr + \sn \sd)$, \normalsize 
	where $\sr := {\rm dim} \: \Ran ( V_ \ell )$ if convergence has occurred.  \label{return_spec}
\ENDFOR
 \end{algorithmic}
\caption{Subspace method to compute a rational eigenvalue closest to a prescribed target}
\label{alg:SM}
\end{algorithm}

Choosing orthonormal bases for ${\mathcal V}_\ell$ and  ${\mathcal W}_\ell$ ensures that the norms of the projected 
matrices do not grow. Moreover, the projection matrices $V_\ell, W_\ell$ after orthonormalization are well-conditioned,
whereas, without orthonormalization, the subspace method above is likely to yield ill-conditioned projection matrices;
see the discussions at the beginning of Section \ref{sec:num_results} regarding the orthogonalization of the bases
for details. These are desirable properties from the numerical point of view. On the other hand, in projection-based 
model order reduction for linear ODE systems, often a bi-orthonormality condition $W_\ell^* V_\ell = I$ is enforced, 
e.\,g., by employing the non-symmetric Lanczos process. This ensures that the projected system is again an ODE. 
The benefit of such a condition in our context is however not clear a priori. In fact, to our best knowledge, an in-depth 
analysis of the influence of the choice of bases on the numerical properties of the projected problem is still missing in the literature. 

The subsequent three subsections of this section establish the quadratic convergence of Algorithm \ref{alg:SM}.
The arguments operate on the singular values of $R(s)$ and $R^{{\mathcal W}_\ell, {\mathcal V}_\ell}(s)$,
especially their smallest singular values. Sections \ref{sec:int_svals} and \ref{sec:analytic_svals} focus on the 
interpolatory properties between these singular values, and the analytical properties of the singular values as 
a function of $s$. Finally, Section \ref{sec:convergence_analysis} deduces the main quadratic convergence 
result by exploiting these interpolatory and analytical properties. 

\subsection{Interpolation of Singular Values}\label{sec:int_svals}
Algorithm \ref{alg:SM} is specifically tailored to satisfy the interpolation properties
\begin{equation}\label{eq:rational_interpolate}
     R(\lambda_k) = R^{{\mathcal W}_\ell, {\mathcal V}_\ell}(\lambda_k) \quad \text{and} \quad
     R^{(j)}(\lambda_k) = \left[ R^{{\mathcal W}_\ell, {\mathcal V}_\ell} \right]^{(j)} (\lambda_k) 
\end{equation}
for $j = 1,\,\dots,\, 2 {\sq} - 1$ and $k = 1,\,\dots,\, \ell$. It is a simple exercise to extend these interpolation
properties to the singular values of $R(\lambda_k)$ and $R^{{\mathcal W}_\ell, {\mathcal V}_\ell}(\lambda_k)$ for $k=1,\,\dots,\,\ell$.


Formally, let us consider the eigenvalues of the matrices
\[
        M(s)
              :=
              R(s)^\ast R(s)
		\quad     \text{and}     \quad
	M^{{\mathcal W}_\ell, {\mathcal V}_\ell}(s)
	     :=
	     	R^{{\mathcal W}_\ell, {\mathcal V}_\ell}(s)^\ast  R^{{\mathcal W}_\ell, {\mathcal V}_\ell}(s)
\]
as functions of $s$ which we denote with $\eta_1(s),\, \dots,\, \eta_{{\sn}}(s)$
and $\eta^{{\mathcal W}_\ell , {\mathcal V}_\ell }_1(s),\, \dots,\, \eta^{{\mathcal W}_\ell, {\mathcal V}_\ell }_{{\sn}}(s)$ and which are sorted in descending order. These eigenvalues
correspond to the squared singular values of the matrices $R(s)$ and $R^{{\mathcal W}_\ell, {\mathcal V}_\ell}(s)$, respectively.
By the definitions of $M(s)$ and $M^{{\mathcal W}_\ell, {\mathcal V}_\ell}(s)$
and exploiting the interpolation properties \eqref{eq:rational_interpolate}, next we deduce the desired interpolation result
concerning the singular values. Throughout the rest of this section, we employ the notations
\[
		\eta'_j(s) := \begin{bmatrix} \frac{\partial \eta_j(s)}{ \partial \Real(s) },& \frac{ \partial \eta_j(s) }{ \partial \Imag(s) } \end{bmatrix}
				\quad	\text{and}		\quad
		\big[ \eta^{{\mathcal W}_\ell , {\mathcal V}_\ell }_j \big]'(s) := 
    \begin{bmatrix} \frac{\partial \eta^{{\mathcal W}_\ell , {\mathcal V}_\ell }_j(s)}{\partial \Real(s)} ,& 
    		\frac{\partial \eta^{{\mathcal W}_\ell ,{\mathcal V}_\ell }_j(s)}{\partial \Imag (s)}
	\end{bmatrix},		
\] 
as well as
\begin{multline*}
	\nabla^2 \eta_j(s)
			:=
			\begin{bmatrix}
					\frac{\partial^2 \eta_j(s)}{ \partial \Real(s)^2 } & \frac{ \partial^2 \eta_j(s) }{ \partial \Real (s) \partial \Imag(s) }	\\
					 \frac{\partial^2 \eta_j(s)}{ \partial \Imag(s) \partial \Real (s) } & \frac{ \partial^2 \eta_j(s) }{ \partial \Imag (s)^2 } 
			\end{bmatrix} 
	\quad	\text{and}		\\
	\nabla^2  \eta^{{\mathcal W}_\ell , {\mathcal V}_\ell }_j (s)
		:=
			\begin{bmatrix}
					\frac{\partial^2 \eta^{{\mathcal W}_\ell , {\mathcal V}_\ell}_j(s)}{ \partial \Real (s)^2 } & 
							\frac{ \partial^2 \eta^{{\mathcal W}_\ell , {\mathcal V}_\ell}_j(s) }{ \partial \Real (s) \partial \Imag (s) }	\\
					 \frac{\partial^2 \eta^{{\mathcal W}_\ell , {\mathcal V}_\ell}_j(s)}{ \partial \Imag (s) \partial \Real (s) } & 
					 		\frac{ \partial^2 \eta^{{\mathcal W}_\ell , {\mathcal V}_\ell}_j(s) }{ \partial \Imag (s)^2 } 
			\end{bmatrix} 
\end{multline*}
for $j = 1,\,\dots,\, \sn$.
\begin{theorem}[Hermite interpolation]\label{thm:sval_Hermite_interpolate}
Regarding Algorithm \ref{alg:SM} with $\sq \ge 2$, the following assertions hold for $k = 1,\, \dots,\, \ell$ and $j = 1,\, \dots \sn$:
\begin{enumerate}
	\item[\bf (i)] It holds that $\eta_j(\lambda_k)	=	\eta^{{\mathcal W}_\ell , {\mathcal V}_\ell }_j(\lambda_k)$.
	\item[\bf (ii)] If $\eta_j(\lambda_k)$ is simple, then also $\eta^{{\mathcal W}_\ell , {\mathcal V}_\ell }_j(\lambda_k)$ is simple.
	In this case,
	\[
		 \eta'_j(\lambda_k)	=	\big[ \eta^{{\mathcal W}_\ell , {\mathcal V}_\ell }_j \big]'(\lambda_k)
		 	\quad	\mathrm{and}		\quad
		 \nabla^2 \eta_j(\lambda_k)	=	\nabla^2 \eta^{{\mathcal W}_\ell , {\mathcal V}_\ell }_j(\lambda_k).
	\]
\end{enumerate}
\end{theorem}

\begin{proof}
\begin{enumerate} 
\item[\bf (i)] The assertion immediately follows from (\ref{eq:rational_interpolate}), since 
$M(\lambda_k) = M^{{\mathcal W}_\ell, {\mathcal V}_\ell}(\lambda_k)$.
\item[\bf (ii)] In Algorithm \ref{alg:SM}, it is required that ${\sq} \geq 2$.
Hence, the assertions follow from \eqref{eq:rational_interpolate}, in particular from
\begin{align*}
	M(\lambda_k) &= M^{{\mathcal W}_\ell, {\mathcal V}_\ell}(\lambda_k), \\ 
	M'(\lambda_k) &= \big[ M^{{\mathcal W}_\ell, {\mathcal V}_\ell} \big]'(\lambda_k), \\ 
	M''(\lambda_k) &= \big[ M^{{\mathcal W}_\ell, {\mathcal V}_\ell} \big]''(\lambda_k),
\end{align*}
by using the analytical formulas for the first and second derivatives of eigenvalue functions
of a Hermitian matrix dependent on a real parameter  \cite{Lancaster1964}. 
\end{enumerate}
\end{proof}

The requirement that $\sq \geq 2$ appears to be essential for quadratic convergence of
the subspace framework. The arguments in the rest of this section establishing quadratic 
convergence does not apply for $\sq = 1$. In practice, we observe slower convergence
that is faster than linear convergence with $\sq = 1$.

\subsection{Analytical Properties of Singular Values}\label{sec:analytic_svals}
At an eigenvalue $\lambda_\ast$ of $R(\cdot)$, we must have $\eta_{{\sn}}(\lambda_\ast) = 0$.
Additionally, throughout the rest of this section, the eigenvalue $\lambda_\ast$ under consideration is assumed to be simple, i.\,e., $\eta_1(\lambda_\ast),\, \dots,\, \eta_{{\sn}-1}(\lambda_\ast) > 0$. There are appealing smoothness properties intrinsic to 
$\eta_1(\cdot), \, \dots, \, \eta_{{\sn}}(\cdot)$ as well as $\eta^{{\mathcal W}_\ell , {\mathcal V}_\ell }_1(\cdot), \, \dots, \, \eta^{{\mathcal W}_\ell , {\mathcal V}_\ell }_{{\sn}}(\cdot)$
in a neighborhood of an eigenvalue $\lambda_\ast$ of $R(\cdot)$, as long as the following assumption holds.
\begin{assumption}[Non-defectiveness]\label{assume:non_defect}
Let $\lambda_\ast$ be a simple eigenvalue of $R(\cdot)$ such that, for a given $\xi > 0$, we have
\begin{equation}\label{eq:non_degeneracy}
      \sigma_{\min}(A - \lambda_\ast I_\sk) \geq \xi		\quad     \text{and}   \quad  	
      					\sigma_{\min}( W^\ast_\ell A V_\ell - \lambda_\ast W^\ast_\ell V_\ell ) \geq \xi,
\end{equation}
where $\sigma_{\min}(\cdot)$ denotes the smallest singular value of its matrix argument.
\end{assumption}

An implication of the assumption above, combined with the Lipschitz continuity of the singular value functions, 
is the boundedness of the smallest singular values in \eqref{eq:non_degeneracy} away from zero
in a vicinity of $\lambda_\ast$. Formally,
there exists a neighborhood ${\mathcal N}(\lambda_\ast)$ of $\lambda_\ast$ -- independent of the choice of the subspaces
${\mathcal V}_\ell$ and ${\mathcal W}_\ell$ as long as (\ref{eq:non_degeneracy}) is satisfied -- such that
\begin{equation}\label{eq:uniform_non_degeneracy}
	 \sigma_{\min}(A - s I_\sk) \geq \xi/2		\quad     \text{and}   \quad  	\sigma_{\min}(W^\ast_\ell A V_\ell - s W^\ast_\ell V_\ell) \geq \xi/2
	 \quad \forall s \in {\mathcal N}(\lambda_\ast),
\end{equation}
see the beginning of the proof of Lemma A.1 in \cite{Aliyev2019}. 

The matrix-valued functions $M(\cdot)$ and $M^{{\mathcal W}_\ell, {\mathcal V}_\ell}(\cdot)$ 
are analytic in ${\mathcal N}(\lambda_\ast)$, 
which implies the following smoothness properties that we employ in the next section to analyze the rate of convergence.
The proofs of the first three parts of the result below are straightforward adaptations of those 
for Lemma A.1 and Lemma A.2 in \cite{Aliyev2019}. 
The proof of the fourth part is immediate from the second and third part. In the result and elsewhere, 
$\| \cdot \|_2$ denotes the vector or matrix 2-norm. Moreover, we make use of the notation 
$
    {\mathcal B}(\lambda_\ast, \delta)
$
for the open ball centered at $\lambda_\ast$ with radius $\delta > 0$, that is
\[
       \mathcal B(\lambda_\ast, \delta)
           :=
         \{
               z \in {\mathbb C} \;  | \;   | z - \lambda_\ast | < \delta
         \},
\]
whereas
$
    \overline{\mathcal B}(\lambda_\ast, \delta)
$
denotes the closure of  ${\mathcal B}(\lambda_\ast, \delta)$, that is
the closed ball centered at $\lambda_\ast$ with radius $\delta > 0$, i.\,e.,
\[
        \overline{\mathcal B}(\lambda_\ast, \delta)
           :=
         \{
               z \in {\mathbb C} \;  | \;   | z - \lambda_\ast | \leq \delta
         \}.
\]
By a constant here and in the subsequent arguments in this section, we mean that the scalar
does not depend on $\lambda_j$ for $j = 1,\,\dots,\, \ell$ as well as the subspaces ${\mathcal W}_\ell,\, {\mathcal V}_\ell$.
Rather, it can be expressed fully in terms of the quantities related to the original rational function $R(\cdot)$.
\begin{lemma}\label{thm:anal_prop_svals}
Suppose that Assumption \ref{assume:non_defect} holds, and $\lambda_\ell$ is sufficiently close to the
eigenvalue $\lambda_\ast$ of $R(\cdot)$.
There exist constants $\gamma,\,\delta > 0$ such that $\overline{\mathcal B}(\lambda_\ast,\delta) \subseteq {\mathcal N}(\lambda_\ast)$
satisfying the following assertions:
\begin{enumerate}
	\item[\bf (i)] We have $| \eta_{j}(s) - \eta_{j}(\widehat{s}) | \leq \gamma |s - \widehat{s} |$ and 
		$\big| \eta^{{\mathcal W}_\ell , {\mathcal V}_\ell }_{j}(s) - 
			\eta^{{\mathcal W}_\ell , {\mathcal V}_\ell }_{j}(\widehat{s}) \big| 
											\leq \gamma |s - \widehat{s}|$
		for all $s,\, \widehat{s} \in \overline {\mathcal B}(\lambda_\ast,\delta)$ and for $j = 1,\, \dots,\, \sn$.			
	\item[\bf (ii)] The eigenvalues $\eta_{\sn}(s)$ and $\eta^{{\mathcal W}_\ell, {\mathcal V}_\ell }_{{\sn}}(s)$ are simple
		for all $s \in \overline{\mathcal B}(\lambda_\ast,\delta)$. 
		Hence, the derivatives
		\[
			\frac{\partial \eta_{\sn} ( s )}{\partial s_1 },
			\frac{\partial^2 \eta_{\sn} ( s )}{\partial s_1 \partial s_2 },
			\frac{\partial^3 \eta_{\sn} ( s )}{\partial s_1 \partial s_2 \partial s_3} 
			\;\;		\text{and}		\;\;\;
			\frac{\partial \eta^{{\mathcal W}_\ell , {\mathcal V}_\ell }_{\sn} ( s )}{\partial s_1 },
			\frac{\partial^2 \eta^{{\mathcal W}_\ell , {\mathcal V}_\ell }_{\sn} ( s )}{\partial s_1 \partial s_2 },
			\frac{\partial^3 \eta^{{\mathcal W}_\ell , {\mathcal V}_\ell }_{\sn} ( s )}{\partial s_1 \partial s_2 \partial s_3} 
		\]
		exist for every $s_1,\, s_2,\, s_3 \in \{ \Real (s), \Imag (s) \}$ and 
		for all $s \in {\mathcal B}(\lambda_\ast,\delta)$.
	\item[\bf (iii)] We have
	\[
			\left| \frac{\partial \eta^{{\mathcal W}_\ell , {\mathcal V}_\ell }_{\sn} ( s )}{\partial s_1 }  \right|  \leq \gamma, \quad		
			\left| \frac{\partial^2 \eta^{{\mathcal W}_\ell , {\mathcal V}_\ell }_{\sn} ( s )}{\partial s_1 \partial s_2 }  \right|  \leq \gamma, \quad 
			\left| \frac{\partial^3 \eta^{{\mathcal W}_\ell , {\mathcal V}_\ell }_{\sn} ( s )}{\partial s_1 \partial s_2 \partial s_3}  \right|  \leq \gamma	  
	\]
	for every $s_1,\, s_2,\, s_3 \in \{ \Real (s), \Imag (s) \}$ and for all $s \in \mathcal B(\lambda_\ast,\delta)$.
	\item[\bf (iv)] We have 
	\[
		\left\| \eta'_{\sn}(s) - \eta'_{\sn}(\widehat{s}) \right\|_2 \leq \gamma |s - \widehat{s} |, \quad
		\left\| \big[ \eta^{{\mathcal W}_\ell , {\mathcal V}_\ell }_{\sn} \big]'(s) - \big[ \eta^{{\mathcal W}_\ell , {\mathcal V}_\ell }_{\sn} \big]'(\widehat{s}) \right\|_2 \leq \gamma |s - \widehat{s}|
	\]
	and
	\[
		\left\| \nabla^2 \eta_{\sn}(s) - \nabla^2 \eta_{\sn}(\widehat{s}) \right\|_2 \leq \gamma |s - \widehat{s} |,	\quad
		\left\| \nabla^2 \eta^{{\mathcal W}_\ell , {\mathcal V}_\ell }_{\sn}(s) - \nabla^2 \eta^{{\mathcal W}_\ell , {\mathcal V}_\ell }_{\sn}(\widehat{s}) \right\|_2 \leq \gamma |s - \widehat{s}|
	\]
	for all $s, \widehat{s} \in	{\mathcal B}(\lambda_\ast,\delta)$.
\end{enumerate}
\end{lemma}

\subsection{Convergence Properties}\label{sec:convergence_analysis}
In practice, we observe that Algorithm \ref{alg:SM} nearly always converges to the eigenvalue of $R(\cdot)$ closest to 
the target point $\tau$. Here, we consider two consecutive iterates $\lambda_\ell$, $\lambda_{\ell+1}$ of this 
subspace method, which we assume close to an eigenvalue $\lambda_\ast$ of $R(\cdot)$. Then we prove 
\begin{equation}\label{eq:quad_conv_init}
        | \lambda_{\ell+1} - \lambda_\ast |  \leq  C | \lambda_\ell - \lambda_\ast |^2
\end{equation}
for some constant $C > 0$.   The closeness of $\lambda_\ell$, $\lambda_{\ell+1}$ to $\lambda_\ast$ is a silent
assumption that is kept throughout, even though it is not explicitly stated.
In addition, we deduce the bound \eqref{eq:quad_conv_init} under Assumption~\ref{assume:non_defect}, 
as well as the following assumption.
\begin{assumption}[Non-degeneracy]\label{assume:non_degeneracy}
The Hessian $\nabla^2 \eta_{\sn}(\lambda_\ast)$ is invertible.
\end{assumption}

The main quadratic convergence result relies on the non-singularity of the Hessian of
$\eta^{{\mathcal W}_\ell, {\mathcal V}_\ell}_{\sn}  (\cdot)$ in a ball centered around $\lambda_\ast$.
This is stated formally and proven next.
\begin{lemma}[Uniform non-singularity of the Hessian]\label{lemma:boundedness}
Suppose that Assumptions~\ref{assume:non_defect} and \ref{assume:non_degeneracy} hold. 
Then there exist constants $\alpha,\,\delta > 0$ such that
\begin{equation}\label{eq:bounded_sval}
	\sigma_{\min}( \nabla^2 \eta^{{\mathcal W}_\ell, {\mathcal V}_\ell}_{\sn}  (s) )
          \geq  \alpha
	\quad \forall  s \in {\mathcal B}(\lambda_\ast,\delta).
\end{equation}
\end{lemma}
\begin{proof}
Let $\beta := \sigma_{\min}( \nabla^2 \eta_{\sn}  (\lambda_\ast) )  >  0$. By the Lipschitz
continuity of $\nabla^2 \eta_{\sn}  (\cdot)$ around $\lambda_\ast$ (which follows from part \textbf{(iv)} of Lemma~\ref{thm:anal_prop_svals}), 
there exists a $\widehat{\delta} > 0$ such that $\sigma_{\min}( \nabla^2 \eta_{\sn}  (s) )$ $\geq \beta/2$ for all 
$s \in  {\mathcal B}\big(\lambda_\ast,\widehat{\delta}\big)$. Without loss of generality, we may also assume that
$\nabla^2 \eta_{\sn}^{{\mathcal W}_\ell, {\mathcal V}_\ell}  (\cdot)$ is Lipschitz continuous
in ${\mathcal B}\big(\lambda_\ast,\widehat{\delta}\big)$ with the Lipschitz constant $\gamma$
(once again due to part \textbf{(iv)} of Lemma~\ref{thm:anal_prop_svals}).

Setting $\delta := \min \big\{ \beta / (8 \gamma) , \widehat{\delta} \big\}$, we additionally assume, 
without loss of generality, that $\lambda_\ell \in  {\mathcal B}(\lambda_\ast, \delta)$. 
But then the Hermite interpolation property, 
specifically part \textbf{(ii)} of Theorem \ref{thm:sval_Hermite_interpolate}, implies
\[
		\sigma_{\min}\big( \nabla^2 \eta_{\sn}^{{\mathcal W}_\ell, {\mathcal V}_\ell}  (\lambda_\ell) \big)		=
		 \sigma_{\min}( \nabla^2 \eta_{\sn}  (\lambda_\ell) ) 		\geq 		\beta/2.
\]
Moreover,
\begin{equation*}
\begin{split}
	\big| 
		\sigma_{\min} \big( \nabla^2 \eta_{\sn}^{{\mathcal W}_\ell, {\mathcal V}_\ell}  (\lambda_\ell) \big)
					-
		\sigma_{\min}\big( \nabla^2 \eta_{\sn}^{{\mathcal W}_\ell, {\mathcal V}_\ell}  (s) \big)
	\big|
		&  \leq
	\left\|
		\nabla^2 \eta_{\sn}^{{\mathcal W}_\ell, {\mathcal V}_\ell}  (\lambda_\ell)  -   \nabla^2 \eta_{\sn}^{{\mathcal W}_\ell, {\mathcal V}_\ell}  (s)		
	\right\|_2		\\
		&  \leq			\gamma | \lambda_\ell - s | 	\leq		\beta/4
	\quad 	
\end{split}	
\end{equation*}
for all $s \in \overline{\mathcal B}(\lambda_\ast, \delta)$,
where the first inequality follows from Weyl's theorem \cite[Theorem 4.3.1]{Horn1985}, whereas the second inequality is due to 
the Lipschitz continuity of $\nabla^2 \eta_{\sn}^{{\mathcal W}_\ell, {\mathcal V}_\ell}  (\cdot)$. Hence, we deduce
$\sigma_{\min}\big( \nabla^2 \eta_{\sn}^{{\mathcal W}_\ell, {\mathcal V}_\ell}  (s) \big) \geq \beta/4 =: \alpha$
for all $s \in {\mathcal B}(\lambda_\ast, \delta)$ as desired.
\end{proof}

Now we are ready to present the main quadratic convergence result, where the notation
${\mathcal R}^2 : {\mathbb C} \rightarrow {\mathbb R}^2$ refers to the linear map defined by
$
       {\mathcal R}^2(z) := \begin{bmatrix} \Real (z), \; \Imag (z) \end{bmatrix}.
$ 
\begin{theorem}[Quadratic convergence]\label{thm:quad_conv}
Suppose that Assumptions \ref{assume:non_defect} and \ref{assume:non_degeneracy} hold. Then for the iterates of Algorithm~\ref{alg:SM} with $\sq \geq 2$, 
there exists a constant $C>0$ such that
\eqref{eq:quad_conv_init} is satisfied.
\end{theorem}
\begin{proof}
Let $\delta$ be such that the assertions of Lemmas \ref{thm:anal_prop_svals} and \ref{lemma:boundedness}
hold in the ball ${\mathcal B}(\lambda_\ast,\delta)$.
In particular, the eigenvalues $\eta_{\sn}(\cdot)$ and $\eta_{\sn}^{{\mathcal W}_\ell, {\mathcal V}_\ell}(\cdot)$ are simple,
$\nabla^2 \eta_{\sn}(\cdot)$ and $\nabla^2 \eta_{\sn}^{{\mathcal W}_\ell, {\mathcal V}_\ell}(\cdot)$
are Lipschitz continuous with Lipschitz constant $\gamma > 0$, and the lower bound \eqref{eq:bounded_sval} is satisfied 
for some constant $\alpha > 0$ in ${\mathcal B}(\lambda_\ast,\delta)$.
Without loss of generality, assume that $\lambda_\ell,\, \lambda_{\ell+1} \in {\mathcal B}(\lambda_\ast,\delta)$.

The iterate $\lambda_{\ell+1}$, by definition, is an eigenvalue of $R^{ {\mathcal W}_\ell, {\mathcal V}_\ell }(\cdot)$,
hence we have $\eta^{{\mathcal W}_\ell , {\mathcal V}_\ell }_{\sn}(\lambda_{\ell+1}) = 0$.
Indeed, $\lambda_{\ell+1}$ is a smooth global minimizer of $\eta^{{\mathcal W}_\ell , {\mathcal V}_\ell }_{\sn}(\cdot)$
(i.\,e., the smoothness follows from part \textbf{(ii)} of Lemma \ref{thm:anal_prop_svals}),
implying also $\big[ \eta^{{\mathcal W}_\ell , {\mathcal V}_\ell }_{\sn} \big]'(\lambda_{\ell+1}) = 0$.

By employing the Lipschitz continuity of $\nabla^2 \eta_{\sn}(\cdot)$ in ${\mathcal B}(\lambda_\ast,\delta)$, we have
\[
        0 = \eta_{\sn}'(\lambda_\ast)  =   
          			\eta_{\sn}'(\lambda_\ell)
   +		\int_0^1  {\mathcal R}^2( \lambda_\ast - \lambda_\ell)  \nabla^2 \eta_{\sn}(\lambda_\ell + t (\lambda_\ast - \lambda_\ell)) \: {\mathrm d}t,
\]
which, by exploiting $\nabla^2 \eta_{\sn}(\lambda_\ell) = \nabla^2 \eta^{{\mathcal W}_\ell, {\mathcal V}_\ell}_{\sn}  (\lambda_\ell)$
(see part \textbf{(ii)} of Theorem \ref{thm:sval_Hermite_interpolate}), could be arranged to
\begin{multline}\label{eq:intermed}
    0  = \eta'_{\sn}(\lambda_\ell)  +  
             {\mathcal R}^2 (\lambda_\ast -  \lambda_\ell) \nabla^2 \eta^{{\mathcal W}_\ell, {\mathcal V}_\ell}_{\sn}  (\lambda_\ell) \\
          +    \int_0^1 {\mathcal R}^2( \lambda_\ast - \lambda_\ell)  
          \left( \nabla^2 \eta_{\sn}(\lambda_\ell + t (\lambda_\ast - \lambda_\ell)) - \nabla^2 \eta_{\sn}(\lambda_\ell) \right)  \: {\mathrm d}t.
\end{multline}
Moreover, by a Taylor expansion of $\eta^{{\mathcal W}_\ell, {\mathcal V}_\ell}_{\sn} (\cdot)$ about $\lambda_\ell$, we obtain
\[
     0 =   \big[ \eta^{{\mathcal W}_\ell , {\mathcal V}_\ell }_{\sn} \big]'(\lambda_{\ell+1})  =   
                \big[ \eta^{{\mathcal W}_\ell , {\mathcal V}_\ell }_{\sn} \big]'(\lambda_{\ell})  + 
                      {\mathcal R}^2 (\lambda_{\ell+1} -  \lambda_\ell) \nabla^2 \eta^{{\mathcal W}_\ell, {\mathcal V}_\ell}_{\sn} (\lambda_\ell)
                    + {\mathcal O} \big( | \lambda_{\ell+1} -  \lambda_\ell |^2 \big),
\]
which, combined with $\big[ \eta^{{\mathcal W}_\ell , {\mathcal V}_\ell }_{\sn} \big]'(\lambda_{\ell}) = \eta'_{\sn}(\lambda_\ell)$
(again due to part \textbf{(ii)} of Theorem \ref{thm:sval_Hermite_interpolate}), give rise to 
\begin{multline}\label{intermed2}
      \eta'_{\sn}(\lambda_\ell)  +  
      		 {\mathcal R}^2 (\lambda_\ast -  \lambda_\ell) \nabla^2 \eta^{{\mathcal W}_\ell, {\mathcal V}_\ell}_{\sn} (\lambda_\ell)  
           \\ = 
     	{\mathcal R}^2 (\lambda_\ast -  \lambda_{\ell+1})\nabla^2 \eta^{{\mathcal W}_\ell, {\mathcal V}_\ell}_{\sn}  (\lambda_\ell) 
        +  {\mathcal O} \big( | \lambda_{\ell+1} -  \lambda_\ell |^2 \big).
\end{multline}

In \eqref{eq:intermed}, by plugging the right-hand side of \eqref{intermed2}, 
then exploiting the Lipschitz continuity of $\nabla^2 \eta_{\sn}(\cdot)$, and 
taking the norm, we deduce
\begin{equation}\label{eq:intermed3}
    \left\|
      		{\mathcal R}^2 (\lambda_\ast -  \lambda_{\ell+1}) \nabla^2 \eta^{{\mathcal W}_\ell, {\mathcal V}_\ell}_{\sn}  (\lambda_\ell)
   \right\|_2
          \leq 
     \frac{\gamma}{2}  | \lambda_{\ell}  -  \lambda_\ast |^2  +  {\mathcal O} \big( | \lambda_{\ell+1} -  \lambda_\ell |^2 \big),
\end{equation}
where $\gamma$ is the Lipschitz constant for $\nabla^2 \eta_{\sn}(\cdot)$.
Finally, by employing  
$$ 
     \sigma_{\min}\big( \nabla^2 \eta^{{\mathcal W}_\ell, {\mathcal V}_\ell}_{\sn}  (\lambda_\ell) \big)
          \geq  \alpha
$$
in (\ref{eq:intermed3})
(which is implied by Lemma \ref{lemma:boundedness}) for $\alpha > 0$, we obtain
\[
   	\alpha | \lambda_{\ell+1} - \lambda_\ast |
        \leq 
     \frac{\gamma}{2}  | \lambda_{\ell}  -  \lambda_\ast |^2  +  {\mathcal O} \big( | \lambda_{\ell+1} -  \lambda_\ell |^2 \big).
\]
The desired inequality (\ref{eq:quad_conv_init}) is now immediate from
$| \lambda_{\ell+1} -  \lambda_\ell |^2
				\leq
	2 \left(   | \lambda_{\ell+1} -  \lambda_\ast |^2  +  \right.$
	$\left. | \lambda_{\ell} -  \lambda_\ast |^2   \right)$.
\end{proof}

\section{General Nonlinear Eigenvalue Problem Setting}\label{sec:NEP}
Inspired by the ideas of the previous section for rational eigenvalue problems,
we present a subspace framework for the more general setting of a 
nonlinear eigenvalue problem of the form \eqref{eq:NEP_int}.
Let us consider $T(\cdot)$ and $T_j$ for $j = 1,\, \dots,\, \kappa$ in \eqref{eq:NEP_int} and \eqref{eq:NEP_int0} 
in the partitioned forms
\begin{equation}\label{eq:partition_T}
	T(s) =
		\begin{bmatrix}
			A(s)	&	B(s)		\\
			C(s)	&	D(s)		\\
		\end{bmatrix} \quad \text{and} \quad
	T_j =
		\begin{bmatrix}
			A_j	&	B_j		\\
			C_j	&	D_j		\\
		\end{bmatrix},
\end{equation}
where $A(s),\, A_j \in {\mathbb C}^{{\sk} \times {\sk}}$, $B(s),\, B_j	\in {\mathbb C}^{{\sk} \times {\sm}}$,
$C(s),\, C_j \in {\mathbb C}^{{\sm} \times {\sk}}$, $D(s),\, D_j \in {\mathbb C}^{{\sm} \times {\sm}}$ for all $s\in \C$ such
that ${\sk} + {\sm} = {\sn}$ and ${\sk} \gg {\sm}$. It is a simple exercise to deduce that every finite eigenvalue 
$\lambda \in {\mathbb C}$ of $T(\cdot)$ that is not an eigenvalue of $A(\cdot)$, is also an eigenvalue of the function
\[
	{\mathcal R}(s)  :=  C(s) A(s)^{-1} B(s)  -  D(s).
\]
Conversely, every finite eigenvalue of ${\mathcal R}(\cdot)$ is an eigenvalue of $T(\cdot)$.

Similar to the rational eigenvalue problem setting, the large-scale nature of ${\mathcal R}(\cdot)$ is hidden 
in the middle factor $A(\cdot)$. Hence, we define the reduced matrix-valued function 
corresponding to ${\mathcal R}(\cdot)$ by
\[
	{\mathcal R}^{{\mathcal W}, {\mathcal V}}(s) \; := \; C^V(s) A^{W,V}(s)^{-1} B^W(s)  -  D(s)
\]
in terms of two subspaces ${\mathcal W},\, {\mathcal V} \subseteq \C^\sk$ of equal dimension, say ${\sr} \ll \sk$, and matrices $W,\,V$ 
whose columns form orthonormal bases for them, where
\begin{align}\label{eq:NEP_defnABC_red}
  \begin{split}
	A^{W,V}(s)  &:= W^\ast A(s) V  =   f_1(s) (W^\ast A_1 V) + \dots + f_\kappa(s) (W^\ast A_\kappa V), \\
	B^W(s)  &:= W^\ast B(s)  =   f_1(s) (W^\ast B_1) + \dots + f_\kappa(s) (W^\ast B_\kappa), \quad \text{and} \\
	C^V(s)   &:= C(s) V  = f_1(s) (C_1 V) + \dots + f_\kappa(s) (C_\kappa V).
 \end{split}	
\end{align}	 
The middle factor $A^{W,V}(\cdot)$ of the reduced matrix-valued function is of dimension ${\sr} \times {\sr}$ and much smaller than $A(\cdot)$.

Again, we benefit from the optimization point of view, that is we consider the minimization problem
\begin{equation}\label{eq:NEP_opt}
			\min_{\lambda \in {\mathbb C}}  \sigma_{\min}({\mathcal R}(\lambda)).
\end{equation}
In particular, assuming that the spectra of $A(\cdot)$ and $T(\cdot)$ are disjoint, the eigenvalue of $T(\cdot)$ closest to a prescribed target 
$\tau \in {\mathbb C}$ is the global minimizer of the optimization problem above closest to $\tau$. At every subspace iteration, 
instead of \eqref{eq:NEP_opt}, we solve
\begin{equation}\label{eq:NEP_optr}
			\min_{\lambda \in {\mathbb C}}  \sigma_{\min}\big({\mathcal R}^{{\mathcal W}, {\mathcal V}}(\lambda)\big).
\end{equation}
Specifically, we determine the global minimizer of ${\mathcal R}^{{\mathcal W}, {\mathcal V}}(\cdot)$
closest to the prescribed target $\tau$. The eigenvalues of ${\mathcal R}^{{\mathcal W}, {\mathcal V}}(\cdot)$ are the same as those of the function
\begin{equation}\label{eq:red_NEP}
	T^{W,V}(s) :=
		\begin{bmatrix}
			A^{W,V}(s)	&	B^W (s)		\\
			C^V (s)		&	D(s)		\\
		\end{bmatrix}
			=
		\begin{bmatrix}
			W^\ast	&	0		\\
			0		&	I_{\sm}	\\
		\end{bmatrix}
			T(s)
		\begin{bmatrix}
			V		&	0		\\
			0		&	I_{\sm}	\\
		\end{bmatrix}
		,
\end{equation}
except possibly those that are the eigenvalues of  $A^{W,V}(\cdot)$. Hence,
to retrieve the global minimizer $\widetilde{\lambda}$ of $\sigma_{\min}\big({\mathcal R}^{{\mathcal W}, {\mathcal V}}(\cdot)\big)$ closest to $\tau$, we find an eigenvalue of $T^{W,V}(\cdot)$ closest to this target point.

We expand the subspaces ${\mathcal W}, {\mathcal V}$ into
$\widetilde{\mathcal W}, \widetilde{\mathcal V}$ so that 
\begin{equation}\label{eq:NEP_Hermite_interpolate}
	{\mathcal R}\big(\widetilde{\lambda}\big) = {\mathcal R}^{\widetilde{\mathcal W}, \widetilde{\mathcal V}}\big(\widetilde{\lambda}\big)	
				\quad  \text{and}  \quad
	{\mathcal R}^{(j)}\big(\widetilde{\lambda}\big) = \left[ {\mathcal R}^{\widetilde{\mathcal W}, \widetilde{\mathcal V}} \right]^{(j)} \big(\widetilde{\lambda}\big)
\end{equation}	
hold for $j = 1,\,\dots,\, {\sq}$ and for a prescribed positive integer ${\sq}$. The following generalization
of Lemma \ref{thm:main_interpolation} indicates how this Hermite interpolation property can
be attained. This result is also a corollary of \cite[Theorem 1]{Gugercin2009}.
\begin{lemma}\label{thm:main_interpolation_gen}
Suppose that $\mu \in {\mathbb C}$ is not an eigenvalue of $A(\cdot)$.
Let $\widetilde{\mathcal W} = {\mathcal W} \oplus {\mathcal W}_\mu$ and 
$\widetilde{\mathcal V} = {\mathcal V} \oplus {\mathcal V}_\mu$, where ${\mathcal V}, {\mathcal W}$
are given subspaces of equal dimension, and ${\mathcal W}_\mu$, ${\mathcal V}_\mu$
are subspaces defined as 
		\[
			{\mathcal V}_\mu := \bigoplus_{j=0}^{{\sq}-1} \Ran \left( \frac{\mathrm{d}^j}{\mathrm{d}s^j} ( A(s)^{-1} B(s) )  \bigg|_{s = \mu} \right),
				\; {\mathcal W}_\mu := \bigoplus_{j=0}^{{\sq}-1} \Ran \left( \frac{\mathrm{d}^j}{\mathrm{d}s^j} \big( C(s) A(s)^{-1} \big)^\ast  \bigg|_{s = \mu} \right)
		\]
for some positive integer ${\sq}$. Let $\widetilde{V}$ and $\widetilde{W}$ be basis matrices of $\widetilde{\mathcal{V}}$ and $\widetilde{\mathcal{W}}$, respectively and assume further that $\widetilde{W}^*A(\mu)\widetilde{V}$ is invertible. Then we have
	\begin{enumerate}
		\item ${\mathcal R}(\mu) = {\mathcal R}^{\widetilde{\mathcal W}, \widetilde{\mathcal V}}(\mu)$, and		
		\item ${\mathcal R}^{(j)}(\mu) = \left[ {\mathcal R}^{\widetilde{\mathcal W}, \widetilde{\mathcal V}} \right]^{(j)} (\mu)$	
		for $j = 1,\,\dots,\, 2{\sq}-1$.
	\end{enumerate}	
\end{lemma}

Based on the discussions and the subspace expansion strategy above, we outline the subspace framework 
to locate the eigenvalue of $T(\cdot)$ closest to the target point $\tau \in {\mathbb C}$ in Algorithm \ref{alg:NEP_SM}.
At iteration $\ell$ of the algorithm, first the subspaces ${\mathcal W}_{\ell-1}, {\mathcal V}_{\ell-1}$
are expanded to ${\mathcal W}_{\ell}, {\mathcal V}_{\ell}$ to achieve Hermite interpolation conditions
at the current candidate $\lambda_\ell$ for the eigenvalue in lines \ref{NEP_exp_start}--\ref{NEP_exp_end}.
Then the next candidate $\lambda_{\ell+1}$ is retrieved by computing the eigenvalue of $T^{W_\ell, V_\ell}(\cdot)$ closest to the target point. 
The termination condition employed at the end in line \ref{NEP_termination} is specified in Section \ref{sec:num_results}
in a way that also sheds light into the choice of the eigenvector estimate $v$ returned in line \ref{NEP_termination}. 

The quick convergence result of Theorem \ref{thm:quad_conv} extends to Algorithm~\ref{alg:NEP_SM} in a straightforward
fashion. Two consecutive iterates $\lambda_\ell,\, \lambda_{\ell+1}$ of Algorithm~\ref{alg:NEP_SM} satisfy
\[
     | \lambda_{\ell + 1} - \lambda_\ast | \leq C | \lambda_\ell  - \lambda_\ast |^2
\]
for some constant $C>0$, provided $\lambda_\ell,\, \lambda_{\ell+1}$ are sufficiently close to an eigenvalue $\lambda_\ast$
and under non-defectiveness and non-degeneracy assumptions analogous to Assumptions~\ref{assume:non_defect}
and~\ref{assume:non_degeneracy}.

\begin{algorithm}[tb]
 \begin{algorithmic}[1]
 
\REQUIRE{matrices $T_1,\,\dots,\,T_{\kappa}  \in \mathbb{C}^{\sn \times \sn}$, 
	meromorphic functions $f_1,\,\dots, \,f_{\kappa} : {\mathbb C} \rightarrow {\mathbb C}$ as in \eqref{eq:NEP_int0},
	partition parameter ${\sm} \in {\mathbb Z}^+$, interpolation parameter ${\sq} \in {\mathbb Z}$ with
	${\sq} \geq 2$, target $\tau \in {\mathbb C}$.}
\ENSURE{estimate $\lambda \in \C$ for the eigenvalue closest to $\tau$ and corresponding eigenvector estimate $v\in\C^{\sn} \setminus \{0\}$.}

\vskip .7ex

\STATE Partition $T(s)$ as in (\ref{eq:partition_T}) so that $A(s) \in {\mathbb C}^{{\sk}\times {\sk}}$, $B(s)	\in {\mathbb C}^{{\sk}\times {\sm}}$, 
$C(s) \in {\mathbb C}^{{\sm}\times {\sk}}$, $D(s) \in {\mathbb C}^{{\sm} \times {\sm}}$ for all $s \in \C$, where ${\sk} := {\sn}-{\sm}$.

\STATE $\lambda_0 \gets \tau$.

\vskip .7ex

\textcolor{mygreen}{\textbf{$\%$ main loop}}
\FOR{$\ell = 0, \, 1,\,2,\,\dots$}
	
	  \STATE $\widetilde{V}_\ell \gets A(\lambda_\ell )^{-1} B(\lambda_\ell)$ and $\widetilde{W}_\ell \gets A(\lambda_\ell )^{-\ast} C(\lambda_\ell)^\ast$.	\label{NEP_exp_start}	
			\FOR{$j = 1,\, \dots,\, {\sq}-1$}
		  \STATE $\widehat{V}_\ell \gets \frac{ \mathrm{d}^j }{ \mathrm{d}s^j } ( A( s )^{-1} B( s ) ) \big|_{s = \lambda_\ell}$ and  
		  		$\widetilde{V}_\ell \gets \begin{bmatrix} \widetilde{V}_{\ell} & \widehat{V}_\ell \end{bmatrix}$.
		  \STATE $\widehat{W}_\ell \gets \frac{ \mathrm{d}^j }{ \mathrm{d}s^j } ( A( s )^{-\ast} C( s )^\ast ) \big|_{s = \lambda_\ell}$
						and
				$\widetilde{W}_\ell \gets \begin{bmatrix} \widetilde{W}_{\ell} & \widehat{W}_\ell \end{bmatrix}$.
		        \ENDFOR 	
		        
	\vskip .7ex	        
		        
	\textcolor{mygreen}{\textbf{$\%$ expand the subspaces to interpolate at $\lambda_\ell$}} 
	\IF{$\ell = 0$}
		\STATE $V_0 \gets \operatorname{orth}\left( \widetilde{V}_0 \right)
						\text{ and }  W_0 \gets \operatorname{orth}\left( \widetilde{W}_0 \right)$.
	\ELSE	        
		\STATE $V_\ell \gets \operatorname{orth}\left(\begin{bmatrix} V_{\ell-1} & \widetilde{V}_\ell \end{bmatrix}\right)
			\text{ and } W_\ell \gets \operatorname{orth}\left(\begin{bmatrix} W_{\ell-1} & \widetilde{W}_\ell \end{bmatrix}\right)$. \label{N_orthogonalize}
	\ENDIF	\label{NEP_exp_end}
	\STATE Form $T^{W_{\ell},V_{\ell}}(s) := 
							\begin{bmatrix}
                              A^{W_{\ell},V_{\ell}}(s)	&	B^{W_{\ell}} (s)		\\
								C^{V_\ell} (s)	&	D(s)		\\
							\end{bmatrix}$,
						where $A^{W_{\ell},V_{\ell}}(\cdot),\, B^{W_{\ell}} (\cdot),\, C^{V_\ell} (\cdot)$	
						are defined as in \eqref{eq:NEP_defnABC_red}.
		
	\vskip .7ex	
						
	\textcolor{mygreen}{\textbf{$\%$ update the eigenvalue estimate}}
	\STATE	$\lambda_{\ell+1}, v_{\ell+1} \gets \text{ the eigenvalue, eigenvector of } T^{W_{\ell}, V_{\ell}}(\cdot)$
			closest to $\tau$.	\label{NEP_inter_point}  
	\vskip .2ex
	\STATE \textbf{Return} $\lambda \gets \lambda_{\ell+1}$ and 
	$
	  v  \gets 	\begin{bmatrix}
			V_\ell	&	0	\\
			0		&	I_{\sm}
		\end{bmatrix} v_{\ell+1} 
	$
	 if convergence has occurred. \label{NEP_termination} 
\ENDFOR
 \end{algorithmic}
\caption{Subspace method to compute a nonlinear eigenvalue closest to a prescribed target}
\label{alg:NEP_SM}
\end{algorithm}

\section{Computing Multiple Eigenvalues}\label{sec:multiple_eig}
The proposed subspace frameworks, Algorithms \ref{alg:SM} and \ref{alg:NEP_SM} for rational eigenvalue
problems and general nonlinear eigenvalue problems, are meant to estimate only one eigenvalue
closest to the prescribed target $\tau$. However, they have natural extensions to compute $k$ eigenvalues
closest to the target for a prescribed integer $k \geq 2$. These extensions are based on extracting multiple 
eigenvalues of the projected problems, and expanding the projection spaces so as to ensure Hermite interpolation
at some of these eigenvalues. 

Before proposing three alternatives for the interpolation points, let us remark
a subtle issue. There is the possibility that some of the eigenvalues of the projected problems $L^{{W}, {V}}(\cdot)$
and $T^{{W},{V}}(\cdot)$ are indeed also the eigenvalues of their top-left blocks  
$L^{W,V}_A(s) := W^\ast A V - s W^\ast V$ and $A^{W,V}(\cdot)$. Even though this situation seems unlikely,
we observe in practice that it sometimes occurs when multiple eigenvalues of the projected problems  
are extracted. We do not take such eigenvalues of $L^{{W}, {V}}(\cdot)$ and $T^{{W},{V}}(\cdot)$ 
into consideration; such eigenvalues may correspond to the poles 
of $R^{{\mathcal W}, {\mathcal V}}(\cdot)$ and ${\mathcal R}^{{\mathcal W}, {\mathcal V}}(\cdot)$ rather than the eigenvalues of 
$R^{{\mathcal W}, {\mathcal V}}(\cdot)$ and ${\mathcal R}^{{\mathcal W}, {\mathcal V}}(\cdot)$.

To summarize, in lines \ref{inter_point} and \ref{NEP_inter_point} of Algorithms \ref{alg:SM} and \ref{alg:NEP_SM}, we choose the interpolation
points for the next iteration from the set $\Lambda^{W_\ell, V_\ell}$ consisting of all (finite) eigenvalues of 
$L^{{W}_\ell, {V}_\ell}(\cdot)$ and $T^{{W}_\ell,{V}_\ell}(\cdot)$ that are not eigenvalues of 
$L^{W_\ell,V_\ell}_A(\cdot)$ and $A^{W_\ell,V_\ell}(\cdot)$. Specifically, we employ one of the following three viable strategies for the selection 
of the interpolation points at the next iteration among $\lambda_{\ell + 1}^{(1)},\, \dots,\, \lambda_{\ell + 1}^{(k)}$,
the $k$ closest to the target point $\tau$ in $\Lambda^{W_\ell, V_\ell}$: 
\begin{enumerate}[leftmargin = 1.65cm]
\item[{\sf ALL:}]\textbf{Interpolate at up to all of the $k$ closest eigenvalues:} 
Hermite interpolation is performed at the next iteration at each $\lambda_{\ell + 1}^{(j)}$
unless the corresponding residual is below the convergence threshold for $j = 1,\, \dots,\, k$;
see (\ref{eq:terminate}) below for the specification of the residual corresponding to $\lambda_{\ell + 1}^{(j)}$.
\item[{\sf BR:}] \textbf{Interpolate at the eigenvalue among the $k$ closest with the best residual:}
Among the points $\lambda_{\ell + 1}^{(1)},\, \dots,\, \lambda_{\ell + 1}^{(k)}$
with residuals greater than the convergence threshold, we choose only the one with the smallest residual for Hermite interpolation 
at the next iteration.
\item[{\sf WR:}] \textbf{Interpolate at the eigenvalue among the $k$ closest with the worst residual:}
We perform Hermite interpolation at only one of $\lambda_{\ell + 1}^{(1)},\, \dots,\, \lambda_{\ell + 1}^{(k)}$, whichever has the largest residual.
\end{enumerate}

\section{One-Sided Variations}\label{sec:one_sided}
Variants of the subspace methods introduced for the rational eigenvalue problems and general nonlinear 
eigenvalue problems are obtained by forming $V_\ell$ as suggested in the proposed frameworks, but 
setting $W_\ell = V_\ell$. Interpolation results in Lemmas \ref{thm:main_interpolation} and  \ref{thm:main_interpolation_gen}
hold even with ${\mathcal V}_\mu$ as stated in those lemmas and ${\mathcal W}_\mu = {\mathcal V}_\mu$, but with the
equality of the derivatives holding in the second parts up to the $(\sq - 1)$st derivative (rather than the $(2\sq - 1)$st 
derivative). In particular, provided $\sq \geq 3$, the interpolation properties between the full and reduced
matrix-valued functions, as well as their first two derivatives are attained at $\mu$. This paves the way for an analysis
analogous to that in Sections \ref{sec:analytic_svals} and \ref{sec:convergence_analysis}, and leads to
an at least quadratic convergence result for the sequences of eigenvalue estimates.

In our experience, these one-sided variants sometimes tend to be quicker. As an example, for a proper 
rational eigenvalue problem $R_{\sp} = C (sI_{\sk} - A)^{-1} B$ with ${\sq} = 3$, one-sided variants expand the 
projection subspace with the directions
\[	
	\begin{bmatrix}
		(A - \lambda_\ell I_{\sk})^{-1} B
			&
		(A - \lambda_\ell I_{\sk})^{-2} B
			&
		(A - \lambda_\ell I_{\sk})^{-3} B
	\end{bmatrix}
	\hskip 6ex
\]
for interpolation at $\lambda_\ell$, while the original two sided subspace method to achieve
quadratic convergence needs to expand the left and right subspaces with the directions
 \[	
 	\begin{bmatrix}
		(A - \lambda_\ell I_{\sk})^{-1} B
			&
		(A - \lambda_\ell I_{\sk})^{-2} B
	\end{bmatrix}
		\quad	\text{and}		\quad
	\begin{bmatrix}
		((A - \lambda_\ell I_{\sk})^\ast)^{-1} C^\ast
			&
		((A - \lambda_\ell I_{\sk})^\ast)^{-2} C^\ast
	\end{bmatrix},
\]
respectively. Both of these expansion tasks require one LU decomposition, but the latter requires
additional back and forward substitutions. Two-sided method also needs to orthogonalize both of the
projection matrices at every subspace iteration, as opposed to orthogonalization of only one 
projection matrix for the one-sided variant. Yet, it appears that the two-sided subspace method
is usually more reliable and numerically more stable in practice.

In what follows, we refer to one-sided variations of {\sf ALL}, {\sf BR}, {\sf WR} as
{\sf ALL1}, {\sf BR1}, {\sf WR1}, respectively.

\section{Numerical Results}\label{sec:num_results}
In this section, we apply the proposed subspace frameworks to several large-scale nonlinear
eigenvalue problems. 
Our implementation and numerical experiments are performed
in Matlab R2020b on an iMac with Mac OS~11.3.1 operating system, Intel\textsuperscript{\textregistered} Core\textsuperscript{\texttrademark}
i5-9600K CPU and 32GB RAM.

In the subsequent three subsections, we present numerical results on proper rational eigenvalue problems
given in the transfer function form, polynomial eigenvalue problems, and three other nonlinear eigenvalue problems
that are neither polynomial nor rational. In these subsections, when reporting the runtime, number of iterations, number of
LU decompositions for a problem, we always run the algorithm five times, and present
the average over the five runs. Before presenting the numerical results, we spell out the
important implementation details below. For the rest, recall that $k$ is the prescribed
number of eigenvalues sought closest to the target point.   

\medskip

\noindent
\textbf{Termination.}
The algorithms are terminated when the norms of the relative residuals of the Ritz pairs associated 
with the $k$ closest eigenvalues of the projected problems are less than a prescribed tolerance
\texttt{tol}. Formally, letting $\lambda^{(j)}_{\ell}$ and $v^{(j)}_{\ell}$ denote the 
$j$th closest eigenvalue of $T^{W_\ell , V_\ell}(\cdot)$ to $\tau$ and corresponding eigenvector
for $j = 1,\,\dots,\, k$, we terminate if
\begin{equation}\label{eq:terminate}
	\Rs\big(\lambda^{(j)}_{\ell}, v^{(j)}_{\ell}\big)
		:=
		\frac{  \left\| T\big( \lambda^{(j)}_{\ell} \big)
				\begin{bmatrix}
					V_\ell	&	0	\\
					0		&	I_{\sm}
				\end{bmatrix}
	v^{(j)}_{\ell} \right\|_\infty   \: \big/  \;   {\big\| v^{(j)}_{\ell} \big\|}_{\infty} }{ \big| f_1\big(\lambda^{(j)}_\ell\big) \big| {\| T_1 \|}_\infty + \dots + \big| f_\kappa\big(\lambda^{(j)}_\ell\big) \big| {\| T_\kappa \|}_\infty }  <  \texttt{tol}
\end{equation} 
for $j = 1,\,\dots,\, k$
for the general nonlinear eigenvalue problem setting of \eqref{eq:NEP_int}. 

A similar termination condition is adopted for the proper rational eigenvalue problems in the transfer 
function form, i.e., $R(s) = R_{\sp}(s) := C(sI_{\sk} - A)^{-1}B v$. To be precise, 
if $\lambda^{(j)}_{\ell}$ and $v^{(j)}_{\ell}$ denote the 
the $j$th closest eigenvalue of $L^{W_\ell , V_\ell}(\cdot)$ to $\tau$ and a corresponding eigenvector
for $j = 1,\,\dots,\, k$, we terminate when
\begin{equation}\label{eq:terminate2}
	\Rs_{\mathsf{r}}\big(\lambda^{(j)}_{\ell}, v^{(j)}_{\ell}\big)
		:=
		\frac{  \left\| L \big( \lambda^{(j)}_{\ell} \big)
				\begin{bmatrix}
					V_\ell	&	0	\\
					0		&	I_{\sn\sd}
				\end{bmatrix}
	v^{(j)}_{\ell} \right\|_\infty \: \big/  \;   {\big\| v^{(j)}_{\ell} \big\|}_{\infty} }{ \big| \lambda^{(j)}_\ell \big|  +					
					\left\|
						\begin{bmatrix}
						A	&	B	\\
						C	&	0
						\end{bmatrix} 					
					\right\|_\infty
					}  <  \texttt{tol}
\end{equation} 
for $j = 1, \dots, k$.  

\medskip

\noindent
\textbf{Initial Subspaces.} 
We require the initial projected matrices $A^{W_0,V_0}(\cdot)$ 
and $W^\ast_0 A V_0$ (in the general nonlinear eigenvalue problem setting and in the proper rational eigenvalue problem setting, respectively)
to be of size $k \times k$ at least.
To make sure this is the case, we form the initial projected problem by interpolating the full problem at the target point $\tau$,
as well as at randomly selected points close to the target point unless otherwise specified.

\medskip

\noindent
\textbf{Orthogonalization of the Bases for the Subspaces.}
The orthogonalization of the bases for the expansion subspaces (i.e., the columns of $\widetilde{V}_\ell, \widetilde{W}_\ell$ in
lines \ref{R_orthogonalize} and \ref{N_orthogonalize} of Algorithms~\ref{alg:SM} and~\ref{alg:NEP_SM}) 
with respect to the existing projection subspaces (spanned by the columns of $V_{\ell-1}, W_{\ell-1}$) via
\[
	\widetilde{V}_\ell  -  V_{\ell-1} \big(V_{\ell-1}^\ast \widetilde{V}_\ell\big),	\quad\quad	
	\widetilde{W}_\ell  -  W_{\ell-1} \big(W_{\ell-1}^\ast \widetilde{W}_\ell\big)
\]
are performed several times in practice. This reorthogonalization strategy seems to improve the stability of the subspace
frameworks, especially close to convergence. In particular, the column spaces of
$(A - s I_{\sk})^{-1} B$ in the case of Algorithm~\ref{alg:SM} and 
$A (s)^{-1} B$ in the case of Algorithm~\ref{alg:NEP_SM} at $s = \lambda_\ell$ and $s = \lambda_{\ell+1}$
are close to each other when $\lambda_{\ell} \approx \lambda_{\ell-1}$. 
Similar remarks also hold for the derivatives of $(A - s I_{\sk})^{-1} B$ and $A (s)^{-1} B$. Hence, the columns of
$\: [ \, V_{\ell-1}  \;\, \widetilde{V}_\ell  \, ] \:$ are nearly linearly dependent, and
these projection matrices are ill-conditioned. Analogously, the 
matrix $\: [ \, W_{\ell-1}  \;\, \widetilde{W}_\ell  \, ] \:$ for left projections 
becomes ill-conditioned
when $\lambda_\ell$ is close to convergence. In our experience, the reorthogonalization strategy, in the presence
of rounding errors, appears to yield well-conditioned projection matrices with nearly orthonormal 
columns. 

\subsection{Proper Rational Eigenvalue Problems}\label{sec:Num_REP}
\subsubsection{A Banded Example}\label{sec:Num_REP_band}
We first employ the variants of Algorithm \ref{alg:SM} to locate several
eigenvalues of a proper rational matrix-valued function $R_{\sp}(s) = C (sI_{\sk} - A)^{-1} B$ closest to given  target points,
where $A \in {\mathbb R}^{10^5\times 10^5}$ is a sparse banded random matrix with bandwidth five, 
and $B \in {\mathbb R}^{10^5\times 2}$, $C \in {\mathbb R}^{2 \times 10^5}$ are also random 
matrices\footnote{The precise matrices $A,\, B,\, C$ for this proper rational eigenvalue problem, as well
as for the two random examples in Section \ref{sec:Num_REP_spar}, are publicly available 
at \url{https://zenodo.org/record/5811971}.}.
We perform experiments with two target points, namely $\tau_1 = -2 + {\rm i}$ and $\tau_2 = 3 - 7 {\rm i}$;
among these two, $\tau_1$ is close to the eigenvalues of $R_{\sp}(\cdot)$,
while $\tau_2$ is away from the eigenvalues of $R_{\sp}(\cdot)$.

\paragraph{Parameters}
We terminate when condition (\ref{eq:terminate2}) is met for \texttt{tol}$ = 10^{-12}$.
The interpolation parameter that determines how many derivatives will be
interpolated is chosen as $\sq = 2$ and $\sq = 5$ for the two-sided and one-sided
variants, respectively.

\paragraph{Estimation of One Eigenvalue}
The iterates $\lambda_\ell$ of Algorithm~\ref{alg:SM} and the corresponding relative
residuals to compute the eigenvalues closest to $\tau_1$ and $\tau_2$ are listed in Table \ref{table:RS1a}
and Table \ref{table:RS1b}, respectively. Note that the stopping criterion is met after 2 and 7 subspace iterations
for $\tau_1$ and $\tau_2$, respectively, and only the last two estimates for $\tau_2$ 
by Algorithm \ref{alg:SM} and the corresponding residuals are given in Table \ref{table:RS1b}.
The results reported in this table are consistent with the quadratic convergence assertion
of Theorem \ref{thm:quad_conv}. A comparison of the runtimes and accuracy of the computed 
results of Algorithm~\ref{alg:SM} and its one-sided variant with \texttt{eigs} is provided in Table \ref{table:RS1_1}.

\begin{table}
\centering
\caption{The iterates and corresponding residuals of Algorithm \ref{alg:SM} to locate the eigenvalue 
of the proper rational matrix-valued function in Section \ref{sec:Num_REP_band} closest to the target point $\tau_1$
and target point $\tau_2$.}
\label{table:RS1}
\subtable[$\; \tau_1 = -2 + {\rm i}$.]{\label{table:RS1a}
\begin{tabular}{c|cl}
$\ell$  & $\lambda_\ell$ 	& 	$\Rs_{\mathsf{r}}(\lambda_\ell, v_\ell)$  \\    
\hline
1 & $-\underline{2.00038}3015796 + \underline{0.9764}43733019{\rm i}$  &  $2.866 \cdot 10^{-7\phantom{\widehat{\Gamma}}}$ \\
2 & $-\underline{2.000388770264} + \underline{0.976430802383}{\rm i}$  &  $2.648 \cdot 10^{-16}$
\end{tabular}}
\medskip
\subtable[$\; \tau_2 = 3 - 7 {\rm i}$.]{\label{table:RS1b}
\begin{tabular}{c|cl}
$\ell$  & $\lambda_\ell$ & $\Rs_{\mathsf{r}}(\lambda_\ell, v_\ell)$  \\    
\hline
6 & $\underline{1.74988954}6050 - \underline{3.1442679}07227{\rm i}$  &  $1.330 \cdot 10^{-9\phantom{\widehat{\Gamma}}}$ \\
7 & $\underline{1.749889549609} - \underline{3.144267920104}{\rm i}$  & $3.263 \cdot 10^{-19}$
\end{tabular}}
\end{table}

\begin{table}
\centering
\caption{Runtimes in seconds of the subspace methods to compute the eigenvalue of the proper 
rational matrix-valued function in Section \ref{sec:Num_REP_band} closest to $\tau_1 =  -2 + {\rm i}$ and $\tau_2 = 3 - 7 {\rm i}$
compared with \texttt{eigs}. The differences in the computed results with \texttt{eigs} in absolute value 
are also reported.}
\label{table:RS1_1}
\begin{tabular}{c|cc|cc}
			&	\multicolumn{2}{c|}{time in s}	&	\multicolumn{2}{c}{difference with \texttt{eigs}}	\\
method		&	$\tau_1$	&	$\tau_2$		&	$\tau_1$			&		$\tau_2$		\\
\hline
two-sided	&	0.55		&	1.90			&	$4 \times 10^{-14}$	&		$3 \times 10^{-14\phantom{\widehat{\Gamma}}}$	\\		
one-sided	&	0.35		&	0.80			&	$4 \times 10^{-9}$ \,	&		$9 \times 10^{-14\phantom{\widehat{\Gamma}}}$	\\[.2em]
\texttt{eigs}	&	4.20		&	5.85			&		--			&			--	
\end{tabular}
\end{table}

\paragraph{Decay in the Residuals when Estimating Multiple Eigenvalues} 
Next we estimate the five eigenvalues closest to the target point $\tau_1$
by employing the variants {\sf ALL}, {\sf BR}, {\sf WR} of Algorithm \ref{alg:SM}.
All of these variants, as well as \texttt{eigs} return exactly the same five closest eigenvalues
up to twelve decimal digits. In order to compare and illustrate the progresses of {\sf ALL}, {\sf WR}, 
{\sf BR}, we present the relative residuals with respect to the number of iterations (until all five eigenvalues 
converge up to the prescribed tolerance) in Table~\ref{table:Multiple_Res}. The eigenvalue estimates corresponding to the residuals typed in blue italic letters are
selected as interpolation points at the next subspace iteration. For all three variants, typically, the relative 
residual of an eigenvalue estimate that is selected as an interpolation point decreases dramatically in the next iteration.

\begin{table}
\caption{The residuals of the iterates of the three variants of Algorithm \ref{alg:SM} to 
compute the five eigenvalues of the proper rational matrix-valued function in Section \ref{sec:Num_REP_band}
closest to $\tau_1 = -2 + {\rm i}$. Interpolation is performed at the eigenvalue estimates whose residuals are typed in blue italic letters.}
\label{table:Multiple_Res}
\subtable[Results for {\sf ALL}.]{
\small
\begin{tabular}{c|ccccc}
$\ell$  &  $\Rs_{\text{r}}\big(\lambda^{(1)}_\ell, v^{(1)}_\ell\big)$   &$\Rs_{\text{r}}\big(\lambda^{(2)}_\ell, v^{(2)}_\ell\big)$  & $\Rs_{\text{r}}\big(\lambda^{(3)}_\ell, v^{(3)}_\ell\big)$  &$\Rs_{\text{r}}\big(\lambda^{(4)}_\ell, v^{(4)}_\ell\big)$ & $\Rs_{\text{r}}\big(\lambda^{(5)}_\ell, v^{(5)}_\ell\big)$ \\
\hline
1 &  \textcolor{blue}{$\mathit{4.66 \cdot 10^{-8\vphantom{\widehat{\Gamma}}}}$}  &  \textcolor{blue}{$\mathit{2.86 \cdot 10^{-6}}$}  &   \textcolor{blue}{$\mathit{1.85 \cdot 10^{-6} \; \:}$} &  \textcolor{blue}{$\mathit{9.95 \cdot 10^{-7}}$}  &  \textcolor{blue}{$\mathit{6.76 \cdot 10^{-6}}$} \\
2 &  $2.42 \cdot 10^{-19}$  &  \textcolor{blue}{$\mathit{7.11 \cdot 10^{-12}}$}  &   \textcolor{blue}{$\mathit{1.44 \cdot 10^{-8} \; \:}$}   &  \textcolor{blue}{$\mathit{2.13 \cdot 10^{-9}}$}   &  \textcolor{blue}{$\mathit{3.48 \cdot 10^{-8}}$} \\
3 &  $3.91 \cdot 10^{-19}$   & $2.48 \cdot 10^{-18}$  &   $2.06 \cdot 10^{-18}$ &  $9.08 \cdot 10^{-19}$  &  $1.85 \cdot 10^{-18}$  
\end{tabular}}
\subtable[Results for {\sf BR}.]{ 
\small \begin{tabular}{c|ccccc}
$\ell$  &  $\Rs_{\text{r}}\big(\lambda^{(1)}_\ell, v^{(1)}_\ell\big)$   &$\Rs_{\text{r}}\big(\lambda^{(2)}_\ell, v^{(2)}_\ell\big)$  & $\Rs_{\text{r}}\big(\lambda^{(3)}_\ell, v^{(3)}_\ell\big)$  &$\Rs_{\text{r}}\big(\lambda^{(4)}_\ell, v^{(4)}_\ell\big)$ & $\Rs_{\text{r}}\big(\lambda^{(5)}_\ell, v^{(5)}_\ell\big)$ \\    
\hline
1 &     \textcolor{blue}{$\mathit{1.36\cdot 10^{-10\vphantom{\widehat{\Gamma}}}}$}     &	$3.10 \cdot 10^{-6}$     &	 $4.11 \cdot 10^{-6}$     &	$3.21 \cdot 10^{-6}$  & $1.65 \cdot 10^{-6}$  \\
2 &     $5.80 \cdot 10^{-19}$     &	\textcolor{blue}{$\mathit{4.40 \cdot 10^{-7}}$}     &	 $3.21 \cdot 10^{-6}$     &	$1.08 \cdot 10^{-6}$  &   $5.70 \cdot 10^{-6}$  \\
3 &     $1.45\cdot 10^{-18}$     &	$5.25 \cdot 10^{-14}$     &	  $5.46 \cdot 10^{-6}$     &	 \textcolor{blue}{$\mathit{1.16 \cdot 10^{-6}}$}  &  $5.89 \cdot 10^{-6}$  \\
4 &     $4.33\cdot 10^{-19}$     &	$2.04 \cdot 10^{-15}$     &	 \textcolor{blue}{$\mathit{1.23 \cdot 10^{-7}}$}     &	$9.23 \cdot 10^{-13}$  &  $2.82 \cdot 10^{-7}$  \\
5 &     $2.83 \cdot 10^{-19}$     &	$8.67 \cdot 10^{-17}$     &	 $9.51 \cdot 10^{-19}$     &	$3.88 \cdot 10^{-14}$  &  \textcolor{blue}{$\mathit{9.41 \cdot 10^{-10}}$}  \\
6 &     $8.17 \cdot 10^{-19}$     &	$9.08 \cdot 10^{-18}$     &	 $9.80 \cdot 10^{-19}$     &	$3.10 \cdot 10^{-15}$  &   $9.42 \cdot 10^{-19}$
\end{tabular}}
\subtable[Results for {\sf WR}.]{ 
\small \begin{tabular}{c|ccccc}
$\ell$  &  $\Rs_{\text{r}}\big(\lambda^{(1)}_\ell, v^{(1)}_\ell\big)$   &$\Rs_{\text{r}}\big(\lambda^{(2)}_\ell, v^{(2)}_\ell\big)$  & $\Rs_{\text{r}}\big(\lambda^{(3)}_\ell, v^{(3)}_\ell\big)$  &$\Rs_{\text{r}}\big(\lambda^{(4)}_\ell, v^{(4)}_\ell\big)$ & $\Rs_{\text{r}}\big(\lambda^{(5)}_\ell, v^{(5)}_\ell\big)$ \\
\hline
3 &     $5.58 \cdot 10^{-10\vphantom{\widehat{\Gamma}}}$     &	$1.33 \cdot 10^{-6}$     &	  $3.64 \cdot 10^{-6}$     &	 $2.21 \cdot 10^{-6}$  &  \textcolor{blue}{$\mathit{4.77 \cdot 10^{-6}}$}  \\
4 &     $4.03 \cdot 10^{-11}$     &	\textcolor{blue}{$\mathit{7.34 \cdot 10^{-8}}$}     &	 $1.22 \cdot 10^{-8}$     &	$7.31 \cdot 10^{-8}$  &  $5.20 \cdot 10^{-11}$  \\
5 &     $2.77 \cdot 10^{-12}$     &	$1.08 \cdot 10^{-18}$     &	 $1.56 \cdot 10^{-9}$     &	\textcolor{blue}{$\mathit{1.95 \cdot 10^{-8}}$}  &  $1.04 \cdot 10^{-11}$  \\
6 &     $6.65 \cdot 10^{-13}$     &	$1.18 \cdot 10^{-18}$     &	 \textcolor{blue}{$\mathit{7.95 \cdot 10^{-11}}$}     &	$7.21 \cdot 10^{-19}$  &   $5.75 \cdot 10^{-13}$  \\  
7 &     $2.60 \cdot 10^{-14}$     &	$2.42 \cdot 10^{-18}$     &	 $1.21 \cdot 10^{-18}$     &	$5.77 \cdot 10^{-19}$  &   $1.51 \cdot 10^{-15}$ 
\end{tabular} \normalsize}
\end{table}

\subsubsection{Comparison of Runtimes on Several Examples}\label{sec:Num_REP_spar}
We also test the subspace methods on 
\textbf{(i)} two randomly generated sparse examples that are not banded, which we refer as \texttt{R1} and \texttt{R2}, 
as well as \textbf{(ii)} the \texttt{Eady} example from the SLICOT benchmark 
collection\footnote{{see} \url{http://slicot.org/20-site/126-benchmark-examples-for-model-reduction}}
for model reduction. The non-banded matrix $A$ in \texttt{R1} and \texttt{R2} is of size $5000$ and $10000$
with about $0.1\%$ and $0.05\%$ nonzero entries, respectively, whereas $B$ and $C$ have two columns 
and two rows, respectively. The matrix $A$ in \texttt{Eady} is of size 598 and dense, while $B$ and $C$ are column and row vectors, respectively.

\paragraph{Parameters} 
In these experiments, we set the termination tolerance $\texttt{tol} = 10^{-8}$. For a fair comparison, we run
\texttt{eigs} also with termination tolerance equal to $10^{-8}$. 
The interpolation parameter (see Algorithm \ref{alg:SM}) is $\sq = 5$ for all of the one-sided variants of the 
subspace method, and $\sq = 2$ for the the two-sided subspace methods excluding the \texttt{Eady} example. In the 
applications of the two-sided subspace methods to the \texttt{Eady} example, we use $\sq = 5$ to 
reduce the number of LU decomposition computations, which is considerably expensive compared to back and 
forward substitutions as this is a dense example.

\paragraph{Comparison of Runtimes} 
Runtimes of the subspace methods and \texttt{eigs} are reported in Table \ref{table:RS2_1}. For the random examples,
the target point $\tau = 3 + 2{\rm i}$ is close to the spectrum, whereas $\tau = 1 + 9 {\rm i}$ is away. 
Similarly, the target points $\tau = -4 + 4{\rm i}$ and $\tau = -100+ 20 {\rm i}$
are close to and away from the spectrum of the \texttt{Eady} example.

All of the computed eigenvalues by the subspace methods and \texttt{eigs} differ by amounts around the prescribed
tolerance $10^{-8}$ with one exception. The fifth closest eigenvalue to  $\tau = 1 + 9 {\rm i}$ for \texttt{R2}
computed by the subspace methods ($0.2352 + 2.2138{\rm i}$) and \texttt{eigs} ($0.8429 + 2.1611{\rm i}$)
differ significantly and are located at a distance 6.8292 and 6.8407 to $\tau = 1 + 9 {\rm i}$, respectively. We have
verified that the absolute residuals at both of these computed eigenvalues are very small. 

The subspace methods on these examples have lower runtimes compared to \texttt{eigs} when the target point is 
away from the spectrum, and there is no notable difference in the runtimes on these examples when the target point is closer to the spectrum.


\begin{table}[h]
\centering
\caption{Runtimes of the subspace methods vs. \texttt{eigs} in seconds on the test examples of Section \ref{sec:Num_REP_spar}
are listed, where $\#$ eigs refers to the number of closest eigenvalues sought.}
\label{table:RS2_1}
\small
\begin{tabular}{l|ccccccc}
example, target, $\#$ eigs		&			{\sf ALL}	&	{\sf BR}	&	{\sf WR}	&	{\sf ALL1}		&	{\sf BR1}	&	{\sf WR1}		& \texttt{eigs} \\			
\hline
\texttt{R1}, $3 + 2{\rm i}$, 	5		&	5.15		&	3.55		&	4.82		&	3.40		&	2.22		&	2.50		&	2.27		\\
\texttt{R1}, $3 + 2{\rm i}$,	10		&	7.85		&	5.79		&	7.25		&	5.50		&	3.17		&	3.82		&	2.90		\\
\texttt{R1}, $1 + 9{\rm i}$,	5		&	7.20		&	5.53		&	7.64		&	4.39		&	4.51		&	4.70		&	9.85		\\	
\texttt{R1}, $1 + 9{\rm i}$,	10		&	11.79	&	11.46	&	10.94	&	8.36		&	7.02		&	5.70		&	14.60		\\	
\hline
\texttt{R2}, $3 + 2{\rm i}$, 	5		&	39.04		&	30.41		&	38.95		&	28.41		&	15.63		&	30.64		&	16.35		\\
\texttt{R2}, $1 + 9{\rm i}$,	5		&	45.79		&	44.15		&	52.21		&	34.21		&	16.46		&	33.96		&	104.44		\\
\hline
\texttt{Eady}, $-4 + 4{\rm i}$, 	5		&	0.13		&	0.15		&	0.13		&	0.13		&	0.26		&	0.17		&	0.39		\\
\texttt{Eady}, $-4 + 4{\rm i}$, 	10		&	0.19		&	0.34		&	0.26		&	0.34		&	0.85		&	0.68		&	0.50		\\
\texttt{Eady}, $-100 + 20{\rm i}$,	5		&	0.23		&	0.35		&	0.29		&	0.15		&	0.13		&	0.13		&	2.24		\\
\texttt{Eady}, $-100 + 20{\rm i}$,	10		&	0.37		&	0.73		&	0.43		&	0.32		&	0.29		&	0.32		&	2.90
\end{tabular}
\normalsize
\end{table}


\subsection{Polynomial Eigenvalue Problems}\label{sec:num_poly}
Next we consider large-scale polynomial eigenvalue problems available in the NLEVP collection \cite{Betcke2012}.
All of these involve quadratic matrix polynomials of the form $P(s) = P_0 + s P_1 + s^2 P_2$
for given square matrices $P_0,\, P_1,\, P_2 \in {\mathbb C}^{{\sn} \times {\sn}}$. 

Throughout this section, the termination condition in (\ref{eq:terminate}) is employed with the
tolerance $\texttt{tol} = 10^{-8}$, the partition parameter in (\ref{eq:partition_T}) is $\sm = 2$, and the interpolation 
parameter (see Algorithm \ref{alg:NEP_SM}) is $\sq = 2$ and $\sq = 3$ for the two-sided and one-sided subspace frameworks, respectively.
The reduced polynomial eigenvalue problems are solved by using a companion form linearization.
In Sections \ref{sec:Schodinger} and \ref{sec:other_Quad_Poly} below, in comparisons of the proposed frameworks 
with CORK, we use the default parameter values for CORK, except its termination tolerance is set equal to $10^{-8}$.
By default, CORK uses only one shift, which is the target point.

\subsubsection{Schrodinger Example}\label{sec:Schodinger}
The first example is the \texttt{schrodinger} example with ${\sn} = 1998$ that arises from a discretization 
of the Schr\"odinger operator. 

\paragraph{Estimation of Multiple Eigenvalues} We compute the $k$ closest eigenvalues to the target point $\tau = -0.36 - 0.001 {\rm i}$ 
for $k = 1,\, \dots,\, 10$ using the variants {\sf ALL}, {\sf BR}, and {\sf WR}
of Algorithm~\ref{alg:NEP_SM}, as well as the latest version of a free Matlab implementation\footnote{available at \url{http://twr.cs.kuleuven.be/research/software/nleps/cork.html}} 
of the CORK algorithm \cite{VanBeeumen2015}. In all cases, the computed eigenvalues by all these methods match exactly up to at least
eight decimal digits. In particular, in Figure~\ref{fig:2a} the eigenvalues near the target point (computed
by applying \texttt{eigs} to a linearization of $P$) are displayed with red crosses, and the ten closest eigenvalues 
computed by the subspace methods and CORK are encircled in blue.

\paragraph{Comparison of Runtimes}
In Figure~\ref{fig:2b}, the runtimes in seconds for the three variants of Algorithm~\ref{alg:NEP_SM}
and the CORK algorithm are plotted as the function of the  prescribed number of eigenvalues. The runtimes appear to be similar, though 
the variants {\sf ALL} and {\sf WR} look slightly faster. 

\paragraph{Decay in the Residuals} For the {\sf ALL} variant of Algorithm \ref{alg:NEP_SM} and to compute the ten eigenvalues closest to $\tau$, 
the termination criterion is satisfied after two iterations. The residuals for the ten eigenvalue estimates at each of these two 
iterations are given in Table~\ref{table:Schrodinger_Res_NLALL}. Interpolation is performed at every one of the ten eigenvalue estimates
at the second iteration, and all residuals decrease dramatically.

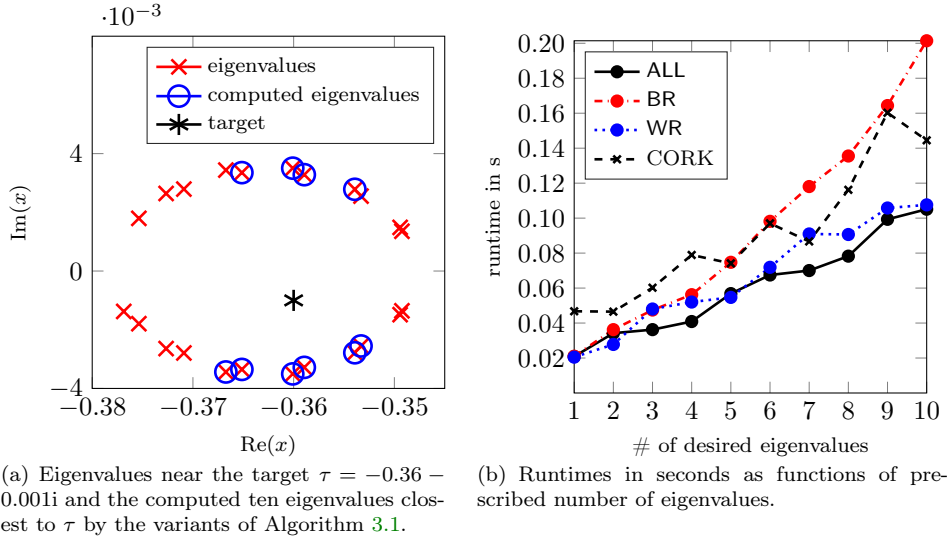
\begin{figure}
	\centering
	\begin{tabular}{cc}
			 \subfigure[Eigenvalues near the target $\tau = -0.36 - 0.001 {\rm i}$ and the computed ten eigenvalues closest 
		to $\tau$ by the variants of Algorithm \ref{alg:NEP_SM}. \label{fig:2a}]{\input{schrodinger_spec.tikz}}	&  
			 \subfigure[Runtimes in seconds as functions of prescribed number
		of eigenvalues. \label{fig:2b}]{\input{runtimes_vs_neig_Schrodinger.tikz}}  
		\end{tabular}
		\caption{The variants of Algorithm~\ref{alg:NEP_SM} applied to the \texttt{schrodinger} example.}
		\label{fig:Schrodinger}
\end{figure}

\begin{table}
\caption{The residuals of the eigenvalue estimates of the {\sf ALL} variant of Algorithm \ref{alg:NEP_SM} 
on the \texttt{schrodinger} example with the target point $\tau = -0.36 - 0.001 {\rm i}$.}
\label{table:Schrodinger_Res_NLALL}
\hskip -.5ex \small
\begin{tabular}{l}
 \begin{tabular}{c|ccccc}
$\ell$  &  $\Rs\big(\lambda^{(1)}_\ell, v^{(1)}_\ell\big)$   & $\Rs\big(\lambda^{(2)}_\ell, v^{(2)}_\ell\big)$  & $\Rs\big(\lambda^{(3)}_\ell, v^{(3)}_\ell\big)$  & $\Rs\big(\lambda^{(4)}_\ell, v^{(4)}_\ell\big)$  & $\:\, \Rs\big(\lambda^{(5)}_\ell, v^{(5)}_\ell\big) \:\,$ \\    
\hline
1 &   \textcolor{blue}{$\mathit{1.05 \cdot 10^{-8\vphantom{\widehat{\Gamma}}}}$}   &	 \textcolor{blue}{$\mathit{1.08 \cdot 10^{-8}}$}     &	  \textcolor{blue}{$\mathit{1.64 \cdot 10^{-8}}$}     &		 \textcolor{blue}{$\mathit{1.77 \cdot 10^{-8}}$}  &    \textcolor{blue}{$\mathit{3.02 \cdot 10^{-8}}$}  \\
2 &     $8.71 \cdot 10^{-16}$     &	 $9.37 \cdot 10^{-16}$    &	 $2.08 \cdot 10^{-15}$     &		$5.70 \cdot 10^{-16}$  &   $1.93 \cdot 10^{-15}$ 
\end{tabular} \medskip \\
\begin{tabular}{c|ccccc}
$\ell$  &  $\Rs\big(\lambda^{(6)}_\ell, v^{(6)}_\ell\big)$   & $\Rs\big(\lambda^{(7)}_\ell, v^{(7)}_\ell\big)$  & $\Rs\big(\lambda^{(8)}_\ell, v^{(8)}_\ell\big)$  & $\Rs\big(\lambda^{(9)}_\ell, v^{(9)}_\ell\big)$ & $\Rs\big(\lambda^{(10)}_\ell, v^{(10)}_\ell\big)$\\  
\hline
1 &      \textcolor{blue}{$\mathit{1.26 \cdot 10^{-8\phantom{\widehat{\Gamma}}}}$}     &		 \textcolor{blue}{$\mathit{4.32 \cdot 10^{-8}}$}     &	  \textcolor{blue}{$\mathit{2.68 \cdot 10^{-8}}$}     &	 \textcolor{blue}{$\mathit{1.38 \cdot 10^{-8}}$}  &   \textcolor{blue}{$\mathit{5.88 \cdot 10^{-8}}$}  \\
2 &     $2.34 \cdot 10^{-15}$     &		$2.83 \cdot 10^{-15}$     & $4.38 \cdot 10^{-15}$     &	$1.34 \cdot 10^{-15}$  &   $2.59 \cdot 10^{-15}$ 
\end{tabular}
\end{tabular} \normalsize
\end{table}

\subsubsection{Other Quadratic Eigenvalue Problems}\label{sec:other_Quad_Poly}
We have also experimented with various other quadratic eigenvalue problems (QEPs) from the NLEVP collection. 
In Table \ref{table:RS3_1}, the runtimes of the variants of Algorithm \ref{alg:NEP_SM} to compute the
five eigenvalues closest to prescribed target points for several QEPs are listed together with the runtimes 
of the CORK algorithm. 
The computed eigenvalues 
by all of the approaches are the same up to nearly eight decimal digits.


\begin{table}[h]
\centering
\caption{Runtimes of the subspace methods and CORK in seconds on several quadratic eigenvalue problems.
The examples marked with (D) are dense examples. The sizes of the problems are $\sn = 2472, 2000, 2000, 100172, 10000, 35955, 1331$
from top to bottom.}
\label{table:RS3_1}
\footnotesize
\begin{tabular}{c|c|ccccccc}
\hskip -.05ex example & target		&		{\sf ALL}	&	{\sf BR}	&	{\sf WR}	&	{\sf ALL1}		&	{\sf BR1}	&	{\sf WR1}	&    CORK    \\			
\hline
\hskip -.05ex \texttt{concrete} & $1 + 5{\rm i}$, 			&	0.12		&	0.07		&	0.07		&	0.04		&	0.06		&	0.06		&	0.07		\\
\hskip -.05ex \texttt{dirac} (D) & $1 + 0.2{\rm i}$				&	2.87		&	3.90		&	3.81		&	2.05		&	2.64		&	2.55		&	13.96		\\
\hskip -2ex \texttt{gen}$\_$\texttt{hyper2} (D) & $3 + 3{\rm i}$	  &	2.37		&	3.49		&	3.48		&	1.98		&	3.58		&	2.93		&	3.77		\\
\hskip -.05ex \texttt{acoustic}$\_$\texttt{wave}$\_$\texttt{2d} & $4 + 0.1{\rm i}$ 		&	4.24		&	4.71		&	3.67		&	3.54		&	2.84		&	2.74		&	3.89		\\
\hskip -.05ex \texttt{pdde}$\_$\texttt{stability} & $-0.1 + 0.03{\rm i}$		&	0.56		&	0.57		&	0.57		&	0.39		&	0.46		&	0.37		&	0.76			\\	
\hskip -.05ex \texttt{railtrack2} & $3 - {\rm i}$				&	23.40		&	20.47		&	18.99		&	17.36		&	10.95		&	12.42		&	11.24		\\
\hskip -.05ex\texttt{utrecht1331} & $300{\rm i}$			&	0.15			&	0.15			&	0.13			&	0.13		&	0.11		&	0.08		&	0.11	
\end{tabular}
\normalsize
\end{table}

\subsubsection{Comparison with a Rational Krylov Method with Adaptive Shifts}\label{sec:adaptive}
The comparisons in the previous two subsections are with CORK that uses a static shift, namely 
the prescribed target point, at every iteration. As Algorithm \ref{alg:NEP_SM} selects new
interpolation points at every iteration, it has similarities with a rational Krylov method for
nonlinear eigenvalue problems that chooses shifts adaptively, where shifts correspond to 
the interpolation points for the polynomial or rational approximation \cite{Beeumen2013, Guttel2014}.  

Here, we compare Algorithm \ref{alg:NEP_SM} with a modification of CORK that uses adaptive 
shifts, equivalently adaptive interpolation points. Specifically, every shift is used a few times.
Then the shift is set equal to the Ritz value closest to the target point. The reason to use 
every shift more than once is to interpolate not only the function values but also the derivatives.
Also, an alternative for the shift selection is to use the Ritz value with the smallest
residual. We have experimented with such alternatives only to observe that they lead to
approaches that are often less reliable.

The results to compute an eigenvalue closest to a target point on a few quadratic eigenvalue problems 
are given in Table \ref{table:adaptive}. In these examples, the computed eigenvalue estimates
by the two approaches are the same up to about four decimal digits.
We have consistently observed that Algorithm \ref{alg:NEP_SM} requires fewer number of
LU decompositions until termination as compared to adaptive 
CORK. A difficulty we have encountered is that our version of adaptive CORK often 
fails to converge to the correct eigenvalue unless the prescribed target point is close 
to an eigenvalue. This is the reason why the target point in these examples are chosen
close to an eigenvalue.

\begin{table}
\centering
\caption{A comparison of Algorithm \ref{alg:NEP_SM} with adaptive CORK on a few
quadratic eigenvalue problems. The computed closest eigenvalues to the target points
are reported in the column of $\lambda$. For both methods, the runtime in seconds
and $\#$ LU decompositions are listed.}
\label{table:adaptive}
\small
\begin{tabular}{ccc|cc|cc}
					&					& 						&\multicolumn{2}{c|}{Alg. \ref{alg:NEP_SM}}	&	\multicolumn{2}{c}{Adap. CORK}	\\
example		 		& target				&	$\lambda$				&	time			&		$\#$ lu				&	time	&		$\#$ lu	\\
\hline
\texttt{concrete} 		&  $-0.2 + 5{\rm i}$		&	$-0.2067 + 5.2350{\rm i}$	&	0.04			&	2 	&		0.26 		&		6.2	\\[.2em]		
\texttt{acoustic}$\_$\texttt{wave}$\_$\texttt{2d}			&  $4$		&	$4.0281 + 0.0429{\rm i}$				&	1.63			&	3.2		&	4.34			&	4.8		\\[.2em]	    
\texttt{pdde}$\_$\texttt{stability}			&  $-0.1$		&	$-0.1028 - 0.0001{\rm i}$					&	0.15			&	3		&	0.29			&	3.6		\\[.2em]		
\texttt{railtrack2}		&  $3 - 1{\rm i}$		&	$3.1433 - 1.7621{\rm i}$					&	10.17			&	3		&	22.87			&	5		\\[.2em]		
\end{tabular}
\end{table}

\subsubsection{The Effect of the Partition Parameter}
We investigate the effect of the partition parameter $\sm$ on the ${\sf ALL}$
variant numerically. In Figure \ref{fig:3a}, the runtimes are reported as $\sm$
varies in $[1,16]$ for three quadratic eigenvalue problems. The runtimes 
do not change much for smaller values of $\sm$, i.e., for $\sm \in [1,4]$.
But then for larger values of $\sm$ the runtimes increase gradually as $\sm$ 
is increased.

This dependence of the runtimes on $\sm$ is partly explained by the
number of LU decomposition and linear system solves performed, which
are depicted in Figure \ref{fig:3b}. As $\sm$ is increased, the number of
subspace iterations decreases slightly initially, possibly since larger values 
of $\sm$ result in more accurate interpolating reduced problems. But
for $\sm \geq 4$ the number of subspace iterations stagnate and do not
decrease anymore. This variation in the number of iterations is directly 
reflected into the number of LU decompositions shown in Figure \ref{fig:3b}.
On the other hand, the number of linear system solves increases
consistently as a function of $\sm$, also visible in Figure \ref{fig:3b}.
For smaller values of $\sm$ the decrease in the number of $LU$ decompositions
is offset by the increase in the number of linear system solves, leading
to a nearly constant dependence of the runtime on $\sm$. But for $\sm \geq 4$,
there is no offset for the increasing cost of linear system solves,
so the runtime increases.

Apart from runtime considerations, there appears to be a second good 
reason to choose $\sm$ small. For larger values of $\sm$, we have occasionally 
witnessed problems with convergence, whereas for smaller values of $\sm$
convergence is almost always guaranteed. This is merely a practical observation
as of now, which we hope to be able to reason in the future.

\begin{figure}
	\centering
	\begin{tabular}{cc}
	\hskip -2ex
			 \subfigure[Runtimes in seconds of the ${\sf ALL}$ variant as a function of $\sm$.
			 \label{fig:3a}]{\input{partition_effect_time.tikz}}	&  
			 \subfigure[Number of LU decompositions and linear system solves by the ${\sf ALL}$ variant as functions of $\sm$
			 on the \texttt{acoustic}$\_$\texttt{wave}$\_$\texttt{2d} example. \label{fig:3b}]{\input{partition_effect_LU_lsys.tikz}}  
		\end{tabular}
		\caption{The effect of $\, \sm \,$ (the partition parameter) to compute the five closest eigenvalues.
		The target points $\tau = -0.36 - 0.001{\rm i} , \, 300{\rm i} , \, 4 + 0.1{\rm i}$ are used for the 
		\texttt{schrodinger}, \texttt{utrecht1331}, \texttt{acoustic}$\_$\texttt{wave}$\_$\texttt{2d} 
		examples, respectively.}
		\label{fig:partition_effect}
\end{figure}

\subsection{Non-Rational, Non-Polynomial Nonlinear Eigenvalue Problems}\label{sec:num_NEP}
Now we apply the {\sf ALL} variant of Algorithm \ref{alg:NEP_SM}
to three nonlinear eigenvalue problems from the NLEVP collection, that are neither
polynomial nor rational.

The termination tolerance in (\ref{eq:terminate}) is $\texttt{tol} = 10^{-8}$ unless otherwise specified,
while the interpolation parameter in Algorithm \ref{alg:NEP_SM} and partition parameter in (\ref{eq:partition_T})  are $\sq = 2$ and $\sm = 2$
throughout this section. The eigenvalues of the reduced problems are computed 
using the Matlab implementation of NLEIGS \cite{Guttel2014} that is available
on the internet\footnote{available at \url{http://twr.cs.kuleuven.be/research/software/nleps/nleigs.html}}.

\subsubsection{The Gun Problem}
The \texttt{gun} problem, originating from modeling of a radio-frequency gun cavity,
concerns the solution of a nonlinear eigenvalue problem of the form
\[
	T(\lambda) v  =   \left( P_0 - \lambda P_1 + {\rm i} \sqrt{\lambda - \sigma_1^2} \, W_1 + {\rm i} \sqrt{\lambda - \sigma_2^2} \, W_2 \right) v   =  0
\]
for given real symmetric matrices $P_0,\, P_1,\, W_1,\, W_2 \in {\mathbb R}^{9956 \times 9956}$ with positive semidefinite $P_0$ and positive definite $P_1$. 
With the parameters $\sigma_1 = 0$, $\sigma_2 = 108.8774$,
the eigenvalues inside the upper half of the disk in the complex plane centered at $250^2$ on the real axis with radius $300^2 - 250^2$ 
are reported in several works in the literature.

\paragraph{Estimation of Eigenvalues} 
Here, we compute the $k$ eigenvalues closest to the target $\tau = (8 + 2.5 {\rm i}) \cdot 10^4$ for $k = 1,\,\dots,\, 10$
inside the specified upper half-disk by employing the {\sf ALL} variant. 
The initial interpolation points are not chosen randomly anymore. Rather, in addition to the target point $\tau$, we employ the following eight interpolation
points initially: $3 \cdot 10^4, 9 \cdot 10^4, (2+{\rm i}) \cdot 10^4, (6+{\rm i}) \cdot 10^4, (10 + {\rm i})\cdot 10^4,  (2+3{\rm i}) \cdot 10^4,  (6+3{\rm i}) \cdot 10^4, (10+3{\rm i}) \cdot 10^4$.

The computed five closest eigenvalues (encircled in blue) together with all eigenvalues (marked with red crosses) 
inside the half-disk are shown in Figure \ref{fig:5a}. In all cases, the computed eigenvalues are the same as those returned
by NLEIGS applied directly to the full problem up to prescribed tolerances. 
Figure \ref{fig:5b} depicts the runtime and LU decompositions required by the {\sf ALL}
variant as functions of $k$. The runtime of the {\sf ALL} variant on this \texttt{gun} example
is mainly affected by the number of LU decompositions, and linear system solves. As a result, the
runtimes and number of LU decompositions vary more or less in harmony in Figure \ref{fig:5b}
as $k$ increases.

\begin{figure}
 \centering
        \begin{tabular}{cc}
			 \subfigure[The eigenvalues of the \texttt{gun} example inside the upper semi-circle,
		and the computed five closest eigenvalues to the target $\tau = (8 + 2.5 {\rm i}) \cdot 10^4$ by the {\sf ALL} variant
		of Algorithm \ref{alg:NEP_SM}. \label{fig:5a}]{\input{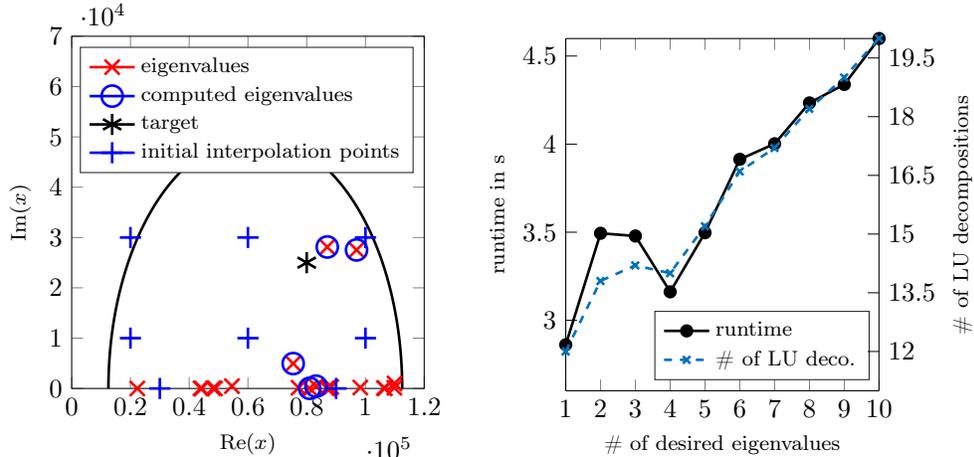}} &	\subfigure[The runtime in seconds and the number of LU decompositions 
		required by the {\sf ALL} variant of Algorithm \ref{alg:NEP_SM} as functions of the
		number of closest eigenvalues sought.\label{fig:5b}]{\input{runtimes_nLU_gun.tikz}}
		\end{tabular}
		\caption{Performance of Algorithm~\ref{alg:NEP_SM} on the \texttt{gun} example.}
		\label{fig:gun}
\end{figure}


\paragraph{Comparison with a Direct Application of NLEIGS}
A direct application of NLEIGS with a tolerance of $10^{-4}$ on
the residual for termination leads to 21 eigenvalues inside the half disk. 
This is also consistent with what is reported in \cite{Beeumen2013}. 
We apply the {\sf ALL} variant of Algorithm~\ref{alg:NEP_SM} 
with tolerance $\texttt{tol} = 10^{-4}$ to compute all 21 
eigenvalues inside the specified region closest to $\tau = 146.71^2$.
The computed eigenvalues are nearly the same with those returned
by NLEIGS with relative differences about $10^{-8}$ or smaller.

Runtimes for both approaches are reported in Table \ref{table:gun_time}.
The total times in the last column are listed disregarding the time
for orthogonalization.  This is because the implementation of NLEIGS does 
not exploit the Kronecker structure of the linearization when orthogonalizing 
the subspaces, and it is likely that orthogonalization costs would be
negligible for NLEIGS just like it is the case for the {\sf ALL variant}, 
had the Kronecker structure been taken into account. Still, the total time
for the {\sf ALL} variant is substantially smaller. The table also reveals
that linear system solves dominate the computation time. The difference
of the total time and time for linear system solves for the {\sf ALL} variant
mainly corresponds to the time spent for the construction of the reduced 
problems and their solution. 

\begin{table}[h]
\centering
\caption{A comparison of the runtimes of the {\sf ALL} variant and NLEIGS. The times for 
linear system solves and orthogonalization, as well as the total time excluding the time 
for orthogonalization are listed in seconds for both of the approaches.}
\label{table:gun_time}
\begin{tabular}{l|ccc}
								&	linear systems		&	orthogonalization	&	total time w/o orth.	\\
\hline
NLEIGS							&	7.89				&		31.03		&		13.18		\\
Algorithm \ref{alg:NEP_SM}, {\sf ALL}	&	2.98				&		0.10			&		4.17	
\end{tabular}
\end{table}

\subsubsection{The Particle in a Canyon Problem}
This problem arises from a finite element discretization of the Schr\"{o}dinger equation
for a particle in a canyon-shaped potential well, and is of the form
\[
	T(\lambda) v 	=	\left( H	-	\lambda I	-	\sum_{k=1}^{81}	e^{{\rm i} \sqrt{2 (\lambda	- \alpha_k)}}	A_k \right) v  =  0, 
\]
where $H \in {\mathbb R}^{16281 \times 16281}$ is sparse, $A_1, \dots, A_{81} \in {\mathbb R}^{16281 \times 16281}$ are sparse with rank two,
and $\alpha_1, \dots, \alpha_k \in {\mathbb R}$ are given branch points.
Indeed, the decompositions of $A_k = L_k U_k^\top$ for $L_k, U_k \in {\mathbb R}^{16281 \times 2}$
are available in NLEVP, but in our implementation we do not exploit this low-rank structure. The eigenvalues of
interest are those on the real axis in the interval $(\alpha_1, \alpha_2)$, where $\alpha_1 \approx -0.1979$ and $\alpha_2 \approx -0.1320$.

\paragraph{Estimation of Eigenvalues} 
We compute the closest eigenvalue and the two closest eigenvalues to the target points $\tau_1 = -0.135$ and $\tau_2 = -0.180$
by the {\sf ALL} variant of the subspace method. Five initial interpolation points are chosen equidistantly
in the interval $(\alpha_1, \alpha_2)$. The results are summarized in Table \ref{table:Particle}.
The computed eigenvalues in the second columns, displayed to a five decimal digit 
accuracy, are the same as those returned by a direct application of NLEIGS. Solutions of the projected subproblems by NLEIGS
take nearly half of the computation time.

\begin{table}[h]
\centering
\caption{Results for the {\sf ALL} variant of Algorithm \ref{alg:NEP_SM} on the particle in a canyon problem.  
In the last three columns, times for linear system solves, subproblems, and total runtimes are listed in seconds.}
\label{table:Particle}
\begin{tabular}{l|cccc}
target,  $\#$ eigs		&	computed eigs					&				linear systems		&	subproblems	&	total time			\\
\hline
$-0.135$,  1				&	$-0.13254$						&		0.37				&		0.90		&		1.98		\\
$-0.135$,  2				&	$-0.13254$, $-0.14060$				&		0.38				&		0.92		&		2.11		\\
\hline
$-0.180$,  1				&	$-0.14060$						&		0.37				&		0.91		&		2.04		\\
$-0.180$,  2				&	$-0.13254$, $-0.14060$				&		0.42				&		0.94		&		2.20	
\end{tabular}
\end{table}

\subsubsection{The Partial Delay Differential Equation Problem}
Our final test is on a problem that comes from a finite difference discretization 
of a delay partial differential equation. This problem is abbreviated as \texttt{Pdde}$\_$\texttt{symmetric}
in the NLEVP collection, and has the form
\[
		T(\lambda) v	=	\left(	F	-	\lambda	I	+	e^{-2\lambda} G	\right) v	=	0
\]
for sparse and banded $F,\,G \in {\mathbb R}^{16129\times 16129}$. The eigenvalues close to zero are of interest.
In particular, we seek eigenvalues in the interval $[-1, 1]$.

\paragraph{Estimation of Eigenvalues}  
We compute the six eigenvalues closest to  $\tau = 0.2$. Five initial interpolation points are
chosen randomly on the real axis from a normal distribution with zero mean and variance equal to 0.2.
The six eigenvalue estimates computed and runtimes are reported in Table \ref{table:PDDE}. All of
the six eigenvalue estimates retrieved have residuals smaller than $5 \cdot 10^{-12}$. The computation
time is once again dominated by the solutions of the projected small-scale eigenvalue problems.

\begin{table}[h]
\centering
\caption{{\sf ALL} variant of Algorithm \ref{alg:NEP_SM} on the partial delay differential equation problem.} 
\label{table:PDDE}
\subtable[Computed closest six eigenvalues to $\tau = 0.2$.]{\label{table:PDDEa}
\begin{tabular}{cccccc}
\hline
$-0.00248$	&	$-0.51908$		&	$-0.56141$		&		$-0.84591$		&		$-0.89726$		&	$-0.92237$	
\end{tabular}}
\subtable[Computation times in seconds.]{\label{table:PDDEb}
\begin{tabular}{ccc}
linear system solves		&	subproblems	&	total runtime	\\
\hline
0.40	&	0.72		&	1.30
\end{tabular}}
\end{table}




\section{Concluding Remarks}
We have proposed subspace frameworks based on Hermite interpolation to deal with the estimation of a few eigenvalues
of a large-scale nonlinear matrix-valued function closest to a prescribed target. At every subspace iteration, first
a reduced nonlinear eigenvalue problem is obtained by employing two-sided or one-sided projections inspired from interpolatory
model-order reduction techniques, then the eigenvalues of the reduced eigenvalue problem are extracted, and finally
the projection subspaces are expanded to attain Hermite interpolation between the full and reduced problem at
the eigenvalues of the reduced problem. We have proven that the proposed framework converges at least
at a quadratic rate in theory under a non-defectiveness and a non-degeneracy assumption.

There are several directions that are open to improvement. One of them is the initial selection of the interpolation points.
This may affect the number of subspace iterations. At the moment, we choose the initial interpolation points randomly 
around the target. A more careful selection of them, for instance using the AAA algorithm \cite{Nakatsukasa2018}, 
may reduce the number of subspace iterations, and improve the reliability. Another issue is the partitioning of the nonlinear matrix-valued
function $T(\cdot)$ as in \eqref{eq:partition_T}, where $B(\cdot)$ has few columns and $C(\cdot)$ has few rows. This partitioning
may affect the convergence and stability properties of the framework. It seems even possible to permute the rows and columns of $T(\cdot)$,
and more generally apply unitary transformations from left or right in advance with the purpose of enhancing the convergence and
stability properties. 
We hope to address these issues in a future work.

\smallskip

\textbf{Software.} Matlab implementations of {\sf ALL}, {\sf BR},  {\sf WR}, {\sf ALL1}, {\sf BR1}, {\sf WR1} 
variants of Algorithms~\ref{alg:SM} and~\ref{alg:NEP_SM}, and 
the banded rational eigenvalue problem example in Section \ref{sec:Num_REP_band}, 
as well as two sparse random rational eigenvalue problems (i.e., \texttt{R1}, \texttt{R2}) in Section \ref{sec:Num_REP_spar}
that are experimented on are publicly available at \url{https://zenodo.org/record/5811971}.

Other nonlinear eigenvalue problem examples on which we perform experiments 
in Sections \ref{sec:num_poly} and \ref{sec:num_NEP} are also publicly available in the NLEVP collection \cite{Betcke2012}.

\smallskip

\textbf{Acknowledgements.}
The authors are grateful to two anonymous referees who provided invaluable comments about
the initial versions of this manuscript.

\bibliography{NEP}

\end{document}

%% file: schrodinger_spec.tikz
%
%
\begin{tikzpicture}

\begin{axis}[%
width=0.36\textwidth,
height=0.36\textwidth,
at={(0in,0in)},
scale only axis,
xmin=-0.38,
xmax=-0.345,
ymin=-0.004,
ymax=0.008,
ytick={-0.004,0,0.004},
axis background/.style={fill=white},
xlabel = {\footnotesize{$\Real(x)$}},
ylabel = {\footnotesize{$\Imag(x)$}},
ylabel near ticks,
xlabel near ticks,
legend style={at={(0,0)}, legend pos=north east, legend cell align=left, align=left, draw=white!15!black}
]
\addplot [color=red, line width=1.0pt, draw=none, mark size=4.0pt, mark=x, mark options={solid, red}]
  table[row sep=crcr]{%
-0.349248897206697	-0.00135470614458134\\
-0.349248897206697	0.00135470614458134\\
-0.349428828807898	-0.00148608665530957\\
-0.349428828807898	0.00148608665530957\\
-0.353308769953765	-0.002552521227971\\
-0.353308769953765	0.002552521227971\\
-0.353924879448936	-0.00278451480609424\\
-0.353924879448936	0.00278451480609424\\
-0.358941868049218	-0.00328245540361044\\
-0.358941868049218	0.00328245540361044\\
-0.360089206081967	-0.0035046042439341\\
-0.360089206081967	0.0035046042439341\\
-0.365139469722035	-0.00335203419385777\\
-0.365139469722035	0.00335203419385777\\
-0.366730309957725	-0.00343883593589287\\
-0.366730309957725	0.00343883593589287\\
-0.370898611263873	-0.00279184826006148\\
-0.370898611263873	0.00279184826006148\\
-0.372663535170539	-0.00264170489780094\\
-0.372663535170539	0.00264170489780094\\
-0.375368745067301	-0.00179406806961886\\
-0.375368745067301	0.00179406806961886\\
-0.376867441553525	-0.00137660804703724\\
};
\addlegendentry{\footnotesize{eigenvalues}}

\addplot [color=blue, line width=1.0pt, draw=none, mark size=4.0pt, mark=o, mark options={solid, blue}]
  table[row sep=crcr]{%
-0.360089206149881	-0.003504604146999\\
-0.358941868085453	-0.00328245506450672\\
-0.358941867914224	0.0032824551447438\\
-0.360089205936328	0.00350460415565419\\
-0.365139469858664	-0.00335203304642359\\
-0.353924880265238	-0.0027845139931042\\
-0.365139469836486	0.00335203256548438\\
-0.353308770162349	-0.00255251803902821\\
-0.353924878905698	0.00278451473775787\\
-0.366730310112902	-0.00343883234041995\\
};
\addlegendentry{\footnotesize{computed eigenvalues}}

\addplot [color=black, line width=1.0pt, draw=none, mark size=4.0pt, mark=asterisk, mark options={solid, black}]
  table[row sep=crcr]{%
-0.36	-0.001\\
};
\addlegendentry{\footnotesize{target}}

\end{axis}
\end{tikzpicture}%

%% file: runtimes_vs_neig_Schrodinger.tikz
%
%
\begin{tikzpicture}

\begin{axis}[%
width=0.36\textwidth,
height=0.36\textwidth,
at={(0in,0in)},
scale only axis,
xmin=1,
xmax=10,
xtick={1,2,3,4,5,6,7,8,9,10},
ymin=0,
ymax=0.2014308146,
ytick={0.02,0.04,0.06,0.08,0.1,0.12,0.14,0.16,0.18,0.2,0.22},
axis background/.style={fill=white},
xlabel = {\footnotesize{\# of desired eigenvalues}},
ylabel = {\footnotesize{runtime in s}},
ylabel near ticks,
xlabel near ticks,
y tick label style={
        /pgf/number format/.cd,
            fixed,
            fixed zerofill,
            precision=2,
        /tikz/.cd
    },
legend style={at={(0,0)}, legend pos=north west, legend cell align=left, align=left, draw=white!15!black}
]
\addplot [color=black, line width=1.0pt, mark size=2.0pt, mark=*, mark options={solid, fill=black, black}]
  table[row sep=crcr]{%
1	0.0208797424\\
2	0.034201965\\
3	0.0362389078\\
4	0.0408659934\\
5	0.0568407254\\
6	0.067472051\\
7	0.0700244114\\
8	0.0782528292\\
9	0.0993236808\\
10	0.1050952752\\
};
\addlegendentry{\footnotesize{\sf ALL}}

\addplot [color=red, dashdotted, line width=1pt, mark size=2.0pt, mark=*, mark options={solid, fill=red, red}]
  table[row sep=crcr]{%
1	0.0210201916\\
2	0.0362325998\\
3	0.0474377736\\
4	0.0562261546\\
5	0.0747460296\\
6	0.0981699476\\
7	0.1180623576\\
8	0.1355287778\\
9	0.164369804\\
10	0.2014308146\\
};
\addlegendentry{\footnotesize{\sf BR}}

\addplot [color=blue, dotted, line width=1.0pt, mark size=2.0pt, mark=*, mark options={solid, fill=blue, blue}]
  table[row sep=crcr]{%
1	0.0205013944\\
2	0.0277675432\\
3	0.0480258362\\
4	0.0520097686\\
5	0.0546725084\\
6	0.0718745704\\
7	0.090957249\\
8	0.0906387262\\
9	0.105790206\\
10	0.1075754574\\
};
\addlegendentry{\footnotesize{\sf WR}}

\addplot [color=black, dashed, line width=1.0pt, mark size=2.0pt, mark=x, mark options={solid, black}]
  table[row sep=crcr]{%
1	0.0467365036\\
2	0.0465204186\\
3	0.0602097058\\
4	0.0789438822\\
5	0.074129838\\
6	0.0970012606\\
7	0.0866676908\\
8	0.1161963662\\
9	0.1602527506\\
10	0.1445022754\\
};
\addlegendentry{\footnotesize{CORK}}

\end{axis}
\end{tikzpicture}%

%% file: partition_effect_time.tikz
%
%
\definecolor{mycolor1}{rgb}{0.71765,0.27451,1.00000}%
\definecolor{mycolor2}{rgb}{0.46667,0.67451,0.18824}%
\begin{tikzpicture}

\pgfplotsset{
width=0.3\textwidth,
height=0.32\textwidth,
at={(0in,0in)},
ylabel near ticks,
xlabel near ticks,
scale only axis,
xmin=1,
xmax=16,
xtick={1,2,3,4,5,6,7,8,9,10,11,12,13,14,15},
xticklabels={{},{2},{},{4},{},{6},{},{8},{},{10},{},{12},{},{14},{}},
xlabel style={font=\color{white!15!black}},
xlabel={partition parameter $\sm$},
separate axis lines,
every outer y axis line/.append style={black},
every y tick label/.append style={font=\color{black}},
every y tick/.append style={black},
}

\begin{axis}[%
axis y line*=right,
ymin=0,
ymax=17,
yminorticks=true,
ylabel={time (\texttt{acoustic\_wave\_2d})},
legend style={at={(0.121,0.677)}, anchor=south west, legend cell align=left, align=left, draw=white!15!black}
]
\addplot [color=red, dashdotted, line width=1.0pt]
  table[row sep=crcr]{%
1	4.6918013924\\
2	4.6622574022\\
3	4.6669524316\\
4	4.8696641846\\
5	5.5601618454\\
6	6.35407769275\\
7	7.1534426456\\
8	7.906006566\\
9	8.8130408862\\
10	9.7947260536\\
11	10.4276922092\\
12	11.3881991475\\
13	12.39934788475\\
14	13.5272676594\\
15	14.3149617888\\
16	14.975199191\\
};

\end{axis}

\begin{axis}[%
axis y line*=left,
every outer y axis line/.append style={black},
every y tick label/.append style={font=\color{black}},
every y tick/.append style={black},
ymin=0,
ymax=1,
yminorticks=true,
ylabel={time (\texttt{schrodinger}, \texttt{utrecht1331})},
legend style={at={(0.02,0.68)}, anchor=south west, legend cell align=left, align=left, draw=white!15!black}
]
\addplot [color=blue, dotted, line width=1.0pt]
  table[row sep=crcr]{%
1	0.1544941724\\
2	0.1557020536\\
3	0.1546389778\\
4	0.167894336\\
5	0.198507207\\
6	0.2248281468\\
7	0.2599903372\\
8	0.3131925474\\
9	0.3687706462\\
10	0.422776131\\
11	0.4254340042\\
12	0.42815913\\
13	0.5021032098\\
14	0.5567568056\\
15	0.7184592102\\
16	0.8188812316\\
};
\addlegendentry{\footnotesize\texttt{utrecht1331}}

\addplot [color=mycolor2, line width=1.0pt]
  table[row sep=crcr]{%
1	0.072144447\\
2	0.0605092704\\
3	0.0641070946\\
4	0.08733640125\\
5	0.0974673826\\
6	0.121497242\\
7	0.149702954666667\\
8	0.179858865\\
9	0.2144531895\\
10	0.2565456408\\
11	0.3101644178\\
12	0.3710693868\\
13	0.465418218\\
14	0.5488116816\\
15	0.648244101\\
16	0.7095725618\\
};
\addlegendentry{\footnotesize\texttt{schrodinger}}

\addplot [color=red, dashdotted, line width=1.0pt]
  table[row sep=crcr]{%
1	-1\\
};
\addlegendentry{\footnotesize\texttt{acoustic\_wave\_2d}}
\end{axis}

\end{tikzpicture}%

%% file: partition_effect_LU_lsys.tikz
%
%
\definecolor{mycolor1}{rgb}{0.30196,0.74510,0.93333}%
\begin{tikzpicture}

\begin{axis}[%
width=0.30\textwidth,
height=0.32\textwidth,
at={(0in,0in)},
ylabel near ticks,
xlabel near ticks,
scale only axis,
xmin=1,
xmax=16,
xtick={2,3,4,5,6,7,8,9,10,11,12,13,14,15,16},
xticklabels={{2},{},{4},{},{6},{},{8},{},{10},{},{12},{},{14},{},{}},
xlabel style={font=\color{white!15!black}},
xlabel={partition parameter $\sm$},
separate axis lines,
every outer y axis line/.append style={black},
every y tick label/.append style={font=\color{black}},
every y tick/.append style={black},
ymin=0,
ymax=12,
axis y line*=left,
yminorticks=true,
ylabel={$\#$ LU decompositions},
yticklabel pos=right,
legend style={at={(0.097,0.738)}, anchor=south west, legend cell align=left, align=left, draw=white!15!black}
]
\addplot [color=blue, line width=1.0pt]
  table[row sep=crcr]{%
1	11\\
2	8.6\\
3	7.6\\
4	7\\
5	7\\
6	7\\
7	7\\
8	7\\
9	7\\
10	7\\
11	7\\
12	7\\
13	7\\
14	7\\
15	7\\
16	7\\
};
\end{axis}

\begin{axis}[%
width=0.30\textwidth,
height=0.32\textwidth,
at={(0in,0in)},
ylabel near ticks,
xlabel near ticks,
scale only axis,
xmin=1,
xmax=16,
xtick={2,3,4,5,6,7,8,9,10,11,12,13,14,15,16},
xticklabels={{2},{},{4},{},{6},{},{8},{},{10},{},{12},{},{14},{},{}},
xlabel style={font=\color{white!15!black}},
xlabel={partition parameter $\sm$},
separate axis lines,
every outer y axis line/.append style={black},
every y tick label/.append style={font=\color{black}},
every y tick/.append style={black},
ymin=0,
ymax=500,
axis y line*=right,
yminorticks=true,
ylabel={$\#$ linear system solves},
yticklabel pos=right,
legend style={at={(0.12,0.02)}, anchor=south west, legend cell align=left, align=left, draw=white!15!black}
]

\addplot [color=blue, line width=1.0pt]
  table[row sep=crcr]{%
1	-1 \\
};
\addlegendentry{\footnotesize $\#$ LU deco.}

\addplot [color=red, dashdotted, line width=1.0pt]
  table[row sep=crcr]{%
1	44\\
2	68.8\\
3	91.2\\
4	112\\
5	140\\
6	168\\
7	196\\
8	224\\
9	252\\
10	280\\
11	308\\
12	336\\
13	364\\
14	392\\
15	420\\
16	448\\
};
\addlegendentry{\footnotesize $\#$ lin. syst. solves}

\end{axis}
\end{tikzpicture}%

%% file: runtimes_nLU_gun.tikz
%
%
\definecolor{mycolor1}{rgb}{0.00000,0.44706,0.74118}%
\begin{tikzpicture}

\begin{axis}[%
width=0.32\textwidth,
height=0.36\textwidth,
at={(0in,0in)},
scale only axis,
xmin=1,
xmax=10,
xtick={1,2,3,4,5,6,7,8,9,10},
separate axis lines,
every outer y axis line/.append style={black},
every y tick label/.append style={font=\color{black}},
every y tick/.append style={black},
ymin=2.6,
ymax=4.5993196408,
ytick={2.5,3,3.5,4,4.5, 5},
xlabel={\footnotesize{\# of desired eigenvalues}},
ylabel={\footnotesize{runtime in s}},
ylabel near ticks,
xlabel near ticks,
axis y line* = left,
axis background/.style={fill=white},
legend style={at={(0,0)}, legend pos=south east, legend cell align=left, align=left, draw=white!15!black}
]
\addplot [color=black, line width=1.0pt, mark size=2.0pt, mark=*, mark options={solid, fill=black, black}]
  table[row sep=crcr]{%
1	2.8611156346 \\
2	3.4938665312 \\
3	3.4783597364 \\
4	3.1616084554 \\
5	3.4971245616 \\
6	3.9137287106 \\
7	4.0020193306 \\
8	4.2339000160  \\
9	4.3386980154 \\
10	4.5993196408 \\
};
\addlegendentry{\footnotesize{runtime}}

\addplot [color=mycolor1, dashed, line width=1.0pt, mark size=2.0pt, mark=x, mark options={solid, mycolor1}]
  table[row sep=crcr]{%
100 100 \\
};
\addlegendentry{\footnotesize{\# of LU deco.}}

\end{axis}

\begin{axis}[%
width=0.32\textwidth,
height=0.36\textwidth,
at={(0in,0in)},
scale only axis,
xmin=1,
xmax=10,
xtick={1,2,3,4,5,6,7,8,9,10},
separate axis lines,
every outer y axis line/.append style={black},
every y tick label/.append style={font=\color{black}},
every y tick/.append style={black},
ymin=11,
ymax=20,
ytick={12,13.5,15,16.5,18,19.5, 21},
ylabel={\footnotesize{\# of LU decompositions}},
ylabel near ticks,
xlabel near ticks,
legend style={at={(0.042,0.752)}, anchor=south west, legend cell align=left, align=left, draw=white!15!black},
hide x axis,
axis y line*=right,
]
\addplot [color=mycolor1, dashed, line width=1.0pt, mark size=2.0pt, mark=x, mark options={solid, mycolor1}]
  table[row sep=crcr]{%
1	12 \\
2	13.8 \\
3	14.2 \\
4	14 \\
5	15.2 \\
6	16.6 \\
7	17.2 \\
8	18.2 \\
9	19 \\
10	20 \\
};

\end{axis}
\end{tikzpicture}%